\documentclass[11pt]{amsart}
\usepackage{amsmath,amstext,amssymb,amsopn,amsthm,mathrsfs}
\usepackage{epsfig,graphics,latexsym,graphicx}
\usepackage[svgnames]{xcolor}
\allowdisplaybreaks

\usepackage{textcase} 

\newtheorem{theorem}{Theorem}[section]
\newtheorem{lemma}[theorem]{Lemma}
\newtheorem{corollary}[theorem]{Corollary}
\newtheorem{proposition}[theorem]{Proposition}
\newtheorem{definition}[theorem]{Definition}
\newtheorem{remark}[theorem]{Remark}       

\numberwithin{equation}{section}

\def\nc{\newcommand}

\nc\m[1]{\left| #1\right|}
\nc\norm[1]{\left\| #1\right\|}
\nc\CC{\mathbb{C}}
\nc\RR{\mathbb{R}}
\nc\QQ{\mathbb{Q}}
\nc\ZZ{\mathbb{Z}}
\nc\NN{\mathbb{N}}

\begin{document}
\title[]{Some capacitary strong type inequalities and related function spaces}
\date{}

\author[Keng Hao Ooi]
{Keng Hao Ooi\textsuperscript{1,2}}
\address{\textsuperscript{1}Department of Mathematics,
National Central University,
No.300, Jhongda Rd., Jhongli City, Taoyuan County 32001, Taiwan (R.O.C.).}
\address{\textsuperscript{2}Department of Mathematics, National Taiwan Normal University, 88, Sec.4, Ting-Chou Road, Taipei, 116059, Taiwan (R.O.C.)}
\email{kooi1@math.ncu.edu.tw}

\author[Nguyen Cong Phuc]
{Nguyen Cong Phuc\textsuperscript{3}\\
\\
\textit{I\MakeTextLowercase{n} \MakeTextLowercase{memory of} 
	P\MakeTextLowercase{rofessor} D\MakeTextLowercase{avid} R. A\MakeTextLowercase{dams}}
}

\address{\textsuperscript{3}Department of Mathematics,
Louisiana State University,
303 Lockett Hall, Baton Rouge, LA 70803, USA.}
\email{pcnguyen@math.lsu.edu}

\maketitle

\begin{abstract} We verify a conjecture of D. R. Adams on a capacitary strong type inequality that generalizes the classical capacitary strong type inequality of V. G. Maz'ya.
	As a result, we characterize related function spaces as   K\"othe duals to a class of  Sobolev multiplier type spaces.  
Moreover, 	using tools from nonlinear potential theory, weighted norm inequalities, and  Banach function space theory, we show that these spaces are also isomorphic to
more concrete spaces that are easy to use and  fit in well with the modern  theory of function spaces of harmonic analysis.

\end{abstract}

\section{Introduction}

Let $n$ be a positive integer.  For $\alpha\in (0,n)$ and  $1<s<\frac{n}{\alpha}$,  the space of Riesz potentials  $\dot{H}^{\alpha, s}=\dot{H}^{\alpha, s}(\mathbb{R}^n)$, is defined  
as the completion of $C_c^\infty(\mathbb{R}^n)$ with respect to the norm 
$$\|u\|_{\dot{H}^{\alpha,s}}=\|(-\Delta)^{\frac{\alpha}{2}}u \|_{L^s(\mathbb{R}^n)}=\|(-\Delta)^{\frac{\alpha}{2}}u \|_{L^s(\mathbb{R}^n)}.$$ 
Here the operator $(-\Delta)^{\frac{\alpha}{2}}$ is understood as
$(-\Delta)^{\frac{\alpha}{2}}:=\mathcal{F}^{-1}(|\xi|^\alpha)\mathcal{F},$ 
where $\mathcal{F}$ is the Fourier transform in $\mathbb{R}^n$. 

It is known that (see \cite{MH}) $\dot{H}^{\alpha,s}$ consists of functions of the form $u=  I_{\alpha}f$ for some $f\in L^s(\mathbb{R}^n)$, and 
$\|u\|_{\dot{H}^{\alpha, s}}=\|f\|_{L^s(\mathbb{R}^n)}.$ Here,   for $\alpha\in (0,n)$ and $x\in\mathbb{R}^n$,
$$I_{\alpha}f(x) := \gamma(n,\alpha) \int_{\mathbb{R}^n} |x-y|^{\alpha-n}  f(y) dy, $$
 is the Riesz potential of $f$ of order $\alpha$, and   
$\gamma(n, \alpha)=\Gamma(\tfrac{n-\alpha}{2})/[\pi^{n/2} 2^\alpha \Gamma(\tfrac{\alpha}{2})].$ Note that the Riesz kernel  $\gamma(n,\alpha) |x|^{\alpha-n}$ is 
 the inverse Fourier transform of $|\xi|^{-\alpha}$ (in the distributional sense).

In the case  $\alpha =k\in \mathbb{N}$ and $1<s< \frac{n}{k}$  we have $\dot{H}^{k,s} \approx \dot{W}^{k,s}$, where $\dot{W}^{k,s}=\dot{W}^{k,s}(\mathbb{R}^n)$ is the homogeneous Sobolev space defined as the completion of $C_c^\infty(\mathbb{R}^n)$ under the norm 
\begin{align*}
\|u\|_{\dot{W}^{k,s}}:=\sum_{|\beta|= k}\|D^{\beta}u\|_{L^{s}(\mathbb{R}^n)}.
\end{align*}

The capacity associated to $\dot{H}^{\alpha,s}$ is the Riesz capacity defined for any set $E\subset\mathbb{R}^n$ by 
\begin{equation*}
{\rm cap}_{\alpha, \,s}(E):=\inf\Big\{\|f\|_{L^{s}(\mathbb{R}^n)}^{s}: f\geq0, I_{\alpha}f\geq 1 {\rm ~on~} E \Big\}.
\end{equation*}

The Choquet integral associated to ${\rm cap}_{\alpha,s}$ of  a function $g:\mathbb{R}^n \rightarrow [0,\infty]$  is defined   by
\begin{equation*}
\int_{\mathbb{R}^n} g d {\rm cap}_{\alpha, s} :=\int_{0}^{\infty}{\rm cap}_{\alpha,s}(\{x\in\mathbb{R}^n: g(x)>t\})dt.
\end{equation*}
This is well defined as long as $g$ is defined quasi-everywhere (q.e.)  with respect to ${\rm cap}_{\alpha, s}$, 
 i.e., $g$ is defined except for a set of zero capacity. With this,  we can define the  Choquet $L^q$ quasi-norm of any ${\rm cap}_{\alpha, s}$ q.e. defined function $u:\mathbb{R}^n \rightarrow [-\infty,\infty]$ by 
$$\|u\|_{L^q({\rm cap}_{\alpha,s})}:= \left(\int_{\mathbb{R}^n} |u|^q d{\rm cap}_{\alpha,s} \right)^{\frac{1}{q}},  \quad q> 0.$$

By definition, a function $u$ is said to be quasi-continuous
with respect to  ${\rm cap}_{\alpha, s}$ if for any $\epsilon>0$ there exists an open set $G$ such that ${\rm cap}_{\alpha,s}(G)<\epsilon$ and $u$ is continuous on $G^c:=\mathbb{R}^n\setminus G$. 
For any $q>0$, the space $L^q({\rm cap}_{\alpha,s})$ and its subspace $\{ f\in L^q({\rm cap}_{\alpha,s}): f {\rm ~ is ~quasicontinuous~w.r.t.~}  {\rm cap}_{\alpha, s}\}$ are both complete quasi-norm spaces. This can be proved as in the proof of \cite[Proposition 2.3]{OP1}. Moreover, they are normable in the case $q\geq1$ (see \cite{OP1} and Theorem \ref{Fnorm} below).

In most applications, the function $u$ under discussion  is given as the potential $u={I}_\alpha  f$ of a function $f$.
One of the fundamental results  in this study is the following capacitary strong type inequality discovered originally by Maz'ya in the early 1960s and subsequently extended by 
Adams and  Dahlberg (see, e.g., \cite{MSh, AH}):
\begin{equation}\label{CSIM}
\int_{\mathbb{R}^n} (I_\alpha  f)^s d{\rm cap}_{\alpha, s} \leq  A \int_{\mathbb{R}^n} f^s dx
\end{equation}
for all nonnegative  $f\in L^s(\mathbb{R}^n)$. The significance of \eqref{CSIM} is that it provides a useful way to control the unwieldy quantity $\|  I_{\alpha} f \|_{L^s({\rm cap}_{\alpha,s})}$. 
In fact, \eqref{CSIM} yields the following important characterization of trace inequalities: Given a nonnegative measure $\mu$, the trace inequality 
\begin{equation*}
\int_{\mathbb{R}^n}  (I_\alpha f)^{s} d\mu \leq C \int_{\mathbb{R}^n} f^{s} dx  
\end{equation*}
holds for all nonnegative  $f\in L^s(\mathbb{R}^n)$ if and only if  $\mu$ satisfies the capacitary condition 
\begin{equation}\label{tr2}
\mu(K)\leq C {\rm cap}_{\alpha,s}(K)
\end{equation}
for all compact sets $K\subset\mathbb{R}^n$.

 After \eqref{CSIM}, due to its relevant to PDEs, spectral theory, and the theory of function spaces,   there have been  interests in controlling the Choquet quasi-norm
$\| I_{\alpha} f  \|_{L^q({\rm cap}_{\alpha,s})}$. In \cite[p. 22]{Ada}, D. R. Adams  conjectured the following bound for this quantity:
\begin{equation}\label{capstrong2}
\int_{\mathbb{R}^n} (I_\alpha  f)^q d{\rm cap}_{\alpha, s} \leq  A \int_{\mathbb{R}^n} f^s (I_\alpha  f)^{q-s}  dx 
\end{equation}
at least for nonnegative functions $f\in L^\infty(\RR^n)$ with compact support and $q\geq 1$. The case $q=s$ goes back to \eqref{CSIM}, and  the  \emph{integer} case where $\alpha \in (0,n)\cap \mathbb{N}$ was obtained by  Adams in the same paper but only in the range
$1\leq q<s+\frac{n}{n-\alpha}$. Most recently, the case $q=1$ and all real $\alpha\in(0,n)$ was settled in  \cite{OP2}.

In this paper, we work out the bound \eqref{capstrong2} for all real $\alpha\in(0,n)$ and all $q\in [1,\infty)$. In fact, our proof here is new even in the case $q=1$.

\begin{theorem}\label{Main1} Let $s>1$, $0<\alpha<\frac{n}{s}$, and $q\in [1,\infty)$. There exists a constant $A>0$ such that 
	\eqref{capstrong2} holds for all nonnegative functions $f\in L^\infty(\RR^n)$ with compact support, or more generally for all nonnegative measurable functions $f$ such that $I_\alpha f \in L^q( {\rm cap}_{\alpha, s})$.
\end{theorem}

We observe that Theorem \ref{Main1} and monotone convergence theorem imply that for $q\geq s$ inequality \eqref{capstrong2} also holds for all nonnegative measurable functions $f$ such that 
\begin{align*}
\int_{\mathbb{R}^n} f^s (I_\alpha  f)^{q-s}  dx<\infty 
\end{align*}
(and there is nothing to prove in the case $\int_{\mathbb{R}^n} f^s (I_\alpha  f)^{q-s}  dx=\infty$).

However, in the case $1\leq q<s$, inequality \eqref{capstrong2} fails for nonnegative measurable functions $f$ that do not decay at infinity such as  $f(x)=e^{|x|^2}$ (provided we interpret $\frac{1}{\infty}=0$). Thus, Theorem \ref{Main1} motivates the search for a large and natural class  of measurable functions $f\geq0$ for which $I_\alpha f \in L^q( {\rm cap}_{\alpha, s})$ so that 
\eqref{capstrong2}  holds for such $f$.  For this purpose, we introduce the following space.
Let $s>1$,   $0<\alpha<\frac{n}{s}$, and $1\leq q<s$. By a weight we mean a nonnegative and locally integrable function in $\RR^n$. We define $\dot{\widetilde{O}}^{\alpha,s}_{q}=\dot{\widetilde{O}}^{\alpha,s}_{q}(\mathbb{R}^n)$ as the space of all measurable functions $g$ in
$\mathbb{R}^n$ such that there exists a  q.e. defined weight $w\in L^q({\rm cap}_{\alpha,s})$ with $\|w\|_{L^q({\rm cap}_{\alpha,s})}\leq 1$   such that    
$$\left(\int_{\mathbb{R}^n} |g|^{s} w^{q-s} dx   \right)^{\frac1s}<+\infty.$$
The integral above 
is understood as $\int_{\{g\not=0\}} |g|^{s} w^{q-s} dx$.
Here the dot indicates that this space is associated to the space of Riesz potentials  $\dot{H}^{\alpha, s}$. 
The space $\dot{\widetilde{O}}^{\alpha,s}_{q}$ will be equipped with the following quasi-norm
$$\|g\|_{\dot{\widetilde{O}}^{\alpha,s}_{q}} := \inf_{w}\left(\int_{\mathbb{R}^n} |g|^{s} w^{q-s} dx   \right)^{\frac1s}, \qquad g\in \dot{\widetilde{O}}^{\alpha,s}_{q},$$
where the infimum is taken over all q.e. defined weights $w$ such that  $\|w\|_{L^q({\rm cap}_{\alpha,s})}\leq 1$.

\begin{theorem}\label{Main2} Let $\alpha\in (0,n)$, $0<s<\frac{n}{\alpha}$, and $q\in [1,s)$. There exists a constant $B>0$ such that the bound
	\begin{equation*}
	\int_{\mathbb{R}^n} (I_\alpha  f)^q d{\rm cap}_{\alpha, s} \leq  B \|f\|^{q}_{\dot{\widetilde{O}}^{\alpha,s}_{q}} 
	\end{equation*}
	holds for all nonnegative functions $f\in \dot{\widetilde{O}}^{\alpha,s}_{q}$. Thus \eqref{capstrong2} holds for all nonnegative functions $f\in \dot{\widetilde{O}}^{\alpha,s}_{q}$.  
\end{theorem}

Theorem \ref{Main2} immediately yields the following result.

\begin{theorem}\label{KV-th} Let $\alpha\in (0,n)$, $0<s<\frac{n}{\alpha}$, and $q\in [1,s)$. There exists a constant $B>0$ such that
	\begin{equation*}
	\int_{\mathbb{R}^n} (I_\alpha  |f|)^q d{\rm cap}_{\alpha, s} \leq  B \|f\|_{\dot{KV}_q}^q. 
	\end{equation*}
Here $\dot{KV}_q=\dot{KV}_q^{\alpha,s}$ is a Kalton-Verbitsky space (see \cite{KV}) defined as the set of all Lebesgue measurable functions $f$ in $\RR^n$ such that    $\|f\|_{\dot{KV}_q}<+\infty$,
where 
\begin{align*}
\|f\|_{\dot{KV}_q}:=\inf\left\{\left(\int_{\mathbb{R}^{n}}h^{s}(I_{\alpha}h)^{q-s}dx\right)^{\frac{1}{q}}:h\in \dot{\widetilde{O}}_{q}^{\alpha,s},~h\geq|f|~{\rm a.e.}\right\}.
\end{align*}	
\end{theorem}

It is not hard to see that the space $\dot{KV}_q$ is continuously embedded into $\dot{\widetilde{O}}^{\alpha,s}_{q}$ for  $1\leq q<s$, $0<\alpha<\frac{n}{s}$. 
In this paper we shall show that the two spaces are indeed isomorphic. We observe that results of this type are known earlier only in the case $q=1$ (see \cite{KV}).

\begin{theorem}\label{KV-q}
	Let $\alpha\in (0,n)$, $0<s<\frac{n}{\alpha}$, and $q\in [1,s)$. Then it holds that
	$$\dot{KV}_q \approx \dot{\widetilde{O}}^{\alpha,s}_{q}.$$
\end{theorem}

However, Theorem \ref{KV-q} does not say much about the nature of $\dot{\widetilde{O}}^{\alpha,s}_{q}$ (or $\dot{KV}_q$). For example, at this point we do not know if the Hardy-Littlewood maximal function as well as  standard Calder\'on-Zygmund type operators are bounded on those spaces.
To this end, we  will  show in Theorem \ref{Main3} below that they are  actually isomorphic to the K\"othe dual of  
a Sobolev multiplier type space. But first we need to provide the definition of Sobolev multiplier type space.

\begin{definition}
	Let $s>1$, $0<\alpha<\frac{n}{s}$, $p>1$, and $0< r \leq  s$. Define  $\dot{M}^{\alpha,s}_{p,r}=\dot{M}^{\alpha,s}_{p,r}(\mathbb{R}^n)$ as the space of 
	functions $f\in L^p_{\rm loc}(\mathbb{R}^n)$
	such that the trace inequality 
	\begin{equation*}
	\left(\int_{\mathbb{R}^n} ({I}_\alpha  h) ^r |f|^p dx\right)^{\frac{1}{p}} \leq C \|h\|_{L^s(\mathbb{R}^n)}^{\frac{r}{p}},
	\end{equation*} 
	holds for all nonnegative functions $h\in L^s(\mathbb{R}^n)$. A norm  of a function $f\in \dot{M}^{\alpha,s}_{p, r}$  is defined as the least possible constant $C$ in the above inequality.
\end{definition}

Note that  the case $r=s$ goes back to $\dot{M}^{\alpha,s}_{p}$ that was considered earlier in \cite{OP1}, i.e.,  one has $\dot{M}^{\alpha,s}_{p,s}=\dot{M}^{\alpha,s}_{p}$. 

With this, we recall that the K\"othe dual
$(\dot{M}^{\alpha, s}_{p,r})'$ of $\dot{M}^{\alpha, s}_{p,r}$ is defined by 
\begin{equation*}
(\dot{M}^{\alpha,s}_{p,r})':=\left\{{\rm measurable~ functions~} g: \sup_{f}\int_{\mathbb{R}^n}|fg|dx<+\infty\right\},
\end{equation*}
where the supremum is taken over all functions $f$ in the unit ball of $\dot{M}^{\alpha,s}_{p,r}$. The norm of a function $g\in (\dot{M}^{\alpha,s}_{p,r})'$ is defined as the above 
supremum.

Let $s>1$,   $0<\alpha<\frac{n}{s}$,  $0< r\leq s$, and $p>1$.  With $p'=\frac{p}{p-1}$, we now define $\dot{N}^{\alpha,s}_{p', s/r}=\dot{N}^{\alpha,s}_{p', s/r}(\mathbb{R}^n)$ as the space of all measurable functions $g$ in
$\mathbb{R}^n$ such that there exists a quasi-continuous weight $w\in A_{1}$ with $\|w\|_{L^{s/r}({\rm cap}_{\alpha,s})}\leq 1$ and $[w]_{A_1}\leq \bar{\bf c}(n,\alpha)$  such that    
$$\left(\int_{\mathbb{R}^n} |g|^{p'} w^{1-p'} dx   \right)^{\frac{1}{p'}}<+\infty.$$
In what follows, for $0\leq w<+\infty$ a.e., the integral above is understood as $\int_{\{g\not=0\}} |g|^{p'} w^{1-p'} dx$.
This implies that $g=0$ a.e. on the set $\{w=0\}$. 
Recall   that $A_1$ is the class of  $A_1$ weights which consists of nonnegative locally integrable functions $w$ in $\mathbb{R}^n$ such that
\begin{equation*} 
{\bf M} w(x)\leq C w(x) \qquad {\rm a.e.}
\end{equation*}
 The $A_1$ characteristic constant of $w$, $[w]_{A_1}$, is defined as the least possible constant $C$ in the above inequality.
The operator ${\bf M}$ stands for the (center)  Hardy-Littlewood maximal function defined for each $f\in L^1_{\rm loc}(\mathbb{R}^n)$ by 
\begin{equation*}
{\bf M} f (x)= \sup_{r>0}   \frac{1}{|B_r(x)|} \int_{B_{r}(x)} |f(y)|dy.
\end{equation*}

A quasi-norm for $\dot{N}^{\alpha,s}_{p', s/r}$ is given by 
$$\|g\|_{\dot{N}^{\alpha,s}_{p', s/r}} := \inf_{w}\left(\int_{\mathbb{R}^n} |g|^{p'} w^{1-p'} dx   \right)^{\frac{1}{p'}}, \quad g\in \dot{N}^{\alpha,s}_{p', s/r},$$
where the infimum is taken over all quasi-continuous weights $w\in A_{1}$ with $\|w\|_{L^{s/r}({\rm cap}_{\alpha,s})}\leq 1$ and $[w]_{A_1}\leq \bar{\bf c}(n,\alpha)$.  
The constant $\bar{\bf c}(n,\alpha)$ depends only on $n, \alpha$, and can be taken as the $A_1$ constant of the Riesz kernel $\gamma(n, \alpha)\left|\cdot\right|^{\alpha-n}$. That is, we can take $\bar{\bf c}(n,\alpha)=[\left|\cdot\right|^{\alpha-n}]_{A_1}$. Note that this constant is sharper than the implicit constant $\bar{\bf c}(n,\alpha,s)$ that was used in \cite{OP1} for the case $r=s$. 

Likewise, we define $\dot{\widetilde{N}}^{\alpha,s}_{p', s/r}=\dot{\widetilde{N}}^{\alpha,s}_{p', s/r}(\mathbb{R}^n)$ as the space of all measurable functions $g$ in
$\mathbb{R}^n$ such that $\|g\|_{\dot{\widetilde{N}}^{\alpha,s}_{p', s/r}}<+\infty$, where now 
$$\|g\|_{\dot{\widetilde{N}}^{\alpha,s}_{p', s/r}} := \inf_{w}\left(\int_{\mathbb{R}^n} |g|^{p'} w^{1-p'} dx   \right)^{\frac{1}{p'}}, \quad g\in \dot{\widetilde{N}}^{\alpha,s}_{p', s/r},$$
with the infimum being taken over all q.e. defined weights $w$ such that  $\|w\|_{L^{s/r}({\rm cap}_{\alpha,s})}\leq 1$.  That is, we do not require quasi-continuity and $A_1$ conditions on the weight $w$ in the definition of $\dot{\widetilde{N}}^{\alpha,s}_{p', s/r}$.

We remark that since  $\norm{w}_{L^\frac{n}{n-\alpha s}(\RR^n)}\leq C \norm{w}_{L^1({\rm cap_{\alpha, s}})}$ (see, e.g., \cite[Equ. (3.1)]{OP3}), it is easy to see from H\"older's inequality that $\dot{N}^{\alpha,s}_{p', s/r}, \dot{\widetilde{N}}^{\alpha,s}_{p', s/r} \hookrightarrow L^1_{\rm loc}(\RR^n)$. Moreover, one can verify that the two are quasi-normed spaces.

It is obvious that for $1\leq q<s$, $$\dot{\widetilde{O}}^{\alpha,s}_{q}=\dot{\widetilde{N}}^{\alpha,s}_{s, s/r},$$ where $r=\frac{s(s-q)}{(s-1)q}$. Similarly, we have 
 $$\dot{O}^{\alpha,s}_{q}=\dot{N}^{\alpha,s}_{s, s/r}, \quad 1\leq q<s, \quad r=\frac{s(s-q)}{(s-1)q},$$ where $\dot{O}^{\alpha,s}_{q}=\dot{O}^{\alpha,s}_{q}(\mathbb{R}^n)$ is the   space of all measurable functions $g$ in  $\mathbb{R}^n$ such that $\|g\|_{\dot{O}^{\alpha,s}_{q}}<+\infty$. The quasi-norm $\|g\|_{\dot{O}^{\alpha,s}_{q}}$ is given by
  $$\|g\|_{\dot{O}^{\alpha,s}_{q}} := \inf_{w}\left(\int_{\mathbb{R}^n} |g|^{s} w^{q-s} dx   \right)^{\frac1s}, \qquad g\in \dot{O}^{\alpha,s}_{q},$$
 where the infimum is taken over all quasi-continuous weights $w$ such that  $\|w\|_{L^q({\rm cap}_{\alpha,s})}\leq 1$ and $[w^{\frac{s-q}{s-1}}]_{A_1}\leq \bar{\bf c}(n,\alpha)$.

\begin{theorem}\label{Main3} Let $s>1$, $0<\alpha<\frac{n}{s}$, $p>1$, and $0< r \leq  s$. Then it holds that, isomorphically,
	$$(\dot{M}^{\alpha, s}_{p,r})' \approx \dot{\widetilde{N}}^{\alpha,s}_{p', s/r} \approx \dot{N}^{\alpha,s}_{p', s/r}.$$
	In particular, for any $1\leq q < s$,
	$$\dot{KV}_q \approx \dot{\widetilde{O}}^{\alpha,s}_{q}\approx \dot{O}^{\alpha,s}_{q} =\dot{N}^{\alpha,s}_{s, \frac{(s-1)q}{s-q}} \approx \left(\dot{M}^{\alpha, s}_{s', \frac{s(s-q)}{(s-1)q}}\right)'$$
\end{theorem}

We remark that Theorem \ref{Main3} is known in the case $r=s$, but under a weaker $A_1$ bound in the definition of $\dot{N}^{\alpha,s}_{p', 1}$ (see \cite{OP1}). Our approach here is different from that of \cite{OP1} and works simultaneously for all $0<r\leq s$.

An application of Theorem \ref{Main3} is that it implies that the Hardy-Littlewood maximal function and Calder\'on-Zygmund type operators are bounded on 
$(\dot{M}^{\alpha, s}_{p, r})'$, $p>1, 0<r\leq s$ and $\dot{\widetilde{O}}^{\alpha,s}_q$, $1\leq q <s$. Also, by Theorem \ref{secondtheorem} below, such operators are bounded on $\dot{M}^{\alpha, s}_{p, r}$.
Moreover, as a consequence of  Theorems \ref{Main2} and \ref{Main3}, we can now obtain other characterizations for the $L^q({\rm cap}_{\alpha,s})$ quasi-norm, when $1\leq q<s$.
For a q.e. defined function $u$ in $\mathbb{R}^n$ we denote by $\dot{\lambda}^{\alpha,s}_q(u)$ and $\dot{\beta}^{\alpha,s}_q(u)$, $s>1$, $0<\alpha  < \frac{n}{s}$, $1\leq q<s$, the following quantities:
$$\dot{\lambda}^{\alpha,s}_q(u):= \inf\left\{\|f\|_{\dot{\widetilde{O}}^{\alpha,s}_{q}}:  0\leq f\in \dot{\widetilde{O}}^{\alpha,s}_{q} \text{ and }  I_\alpha f\geq |u| \text{ q.e.} \right\}$$\
and
$$\dot{\beta}^{\alpha,s}_q(u):= \inf\left\{\left(\int_{\mathbb{R}^n} f^s (I_\alpha f)^{q-s}dx\right)^{\frac1q}:  0\leq f\in \dot{\widetilde{O}}^{\alpha,s}_{q},\,   I_\alpha f\geq |u| \text{ q.e.} \right\}.$$

\begin{theorem}\label{Newnorm2} Let  $s>1$,  $0<\alpha  < \frac{n}{s}$, and $1\leq q<s$. For any q.e. defined function  $u$ in $\mathbb{R}^n$, it holds that 
	\begin{equation*} 
	\|u\|_{L^q({\rm cap}_{\alpha,s})}\simeq \dot{\lambda}^{\alpha,s}_q(u) \simeq \dot{\beta}^{\alpha,s}_q(u).
	\end{equation*}
\end{theorem}

In Section \ref{triplet} below, we revisit the space $\dot{KV}_q$, $1\leq q<s$, and show that it is indeed a Banach function space in the sense of \cite{Lux}.  Moreover, we also discuss the duality triplet 
$$\overline{C_{c}(\mathbb{R}^{n})}^{\dot{M}_{p,r}^{\alpha,s}}\text{--}(\dot{M}_{p,r}^{\alpha,s})'\text{--}\dot{M}_{p,r}^{\alpha,s},$$
where $\overline{C_{c}(\mathbb{R}^{n})}^{\dot{M}_{p,r}^{\alpha,s}}$ indicates the closure of the set of continuous functions with compact support  in $\dot{M}_{p,r}^{\alpha,s}$. That is, we point out that 
\begin{equation}\label{2dual}
  \left(\overline{C_{c}(\mathbb{R}^{n})}^{\dot{M}_{p,r}^{\alpha,s}}\right)^*=(\dot{M}_{p,r}^{\alpha,s})', \quad {\rm and~} \left[(\dot{M}_{p,r}^{\alpha,s})'\right]^* = \dot{M}_{p,r}^{\alpha,s}.
  \end{equation}

To conclude this section, we remark that all of the above results are also available in the inhomogeneous case, where the  space of Riesz potentials   
 $\dot{H}^{\alpha, s}(\mathbb{R}^n)$ and Riesz capacity ${\rm cap}_{\alpha,s}$ are replaced with the space of Bessel potentials $H^{\alpha, s}(\mathbb{R}^n)$ and Bessel capacity ${\rm Cap}_{\alpha,s}$, $s>1, 0<\alpha\leq \frac{n}{s}$, respectively; see Section \ref{inhomo} for the details.

\vspace{.2in}
\noindent {\bf Notation.} In the above and in what follows, for two quasi-normed spaces $F$ and $G$ we write $F\approx G$ (respectively, $F=G$) to indicate that the two spaces are isomorphic (respectively, isometrically isomorphic). For two quantities $A$ and $B$, we write $A\simeq B$ to mean that there exist positive constants $c_1$ and $c_2$ such that $c_1 A\leq B\leq c_2 A$.

\section{Banach function spaces}

 Most of the spaces of this work are   Banach function spaces or are isomorphic to Banach function spaces. We shall adopt the definition of Banach function spaces used in \cite{Lux} (see also \cite{LN}). A Banach function space $X$ on $\mathbb{R}^n$, with Lebesgue measure as the underlying measure, is the set of all Lebesgue measurable functions $f$ in $\mathbb{R}^n$ such that 
$\|f\|_{X}:=\rho(|f|)$ is finite. Here $\rho(f)$, $f\geq 0$, is a given metric function ($0\leq \rho(f)\leq \infty$) that obeys the following properties:

\vspace{.1in}
(P1) $\rho(f)=0$ if and only if $f(x)=0$ a.e. in $\mathbb{R}^n$; $\rho(f_1+f_2)\leq \rho(f_1)+ \rho(f_2)$; and  $\rho( \lambda f)= \lambda \rho(f)$ for any constant $\lambda\geq 0$.

(P2) (Fatou's property) If $\{f_j\}$, $j=1,2, \dots$, is a sequence of nonnegative measurable functions and $f_j \uparrow f$ a.e. in $\mathbb{R}^n$, then $\rho(f_j) \uparrow \rho(f).$

(P3) If   $E$ is any bounded  and measurable subset of $\mathbb{R}^n$, and $\chi_E$ is its characteristic
function, then $\rho(\chi_E)<+\infty$.

(P4) For every bounded and measurable subset  $E$ of $\mathbb{R}^n$, there exists a finite constant
$A_E\geq 0$ (depending only on the set $E$) such that $\int_{E} f dx \leq A_E \rho(f)$
for any nonnegative measurable function $f$ in $\mathbb{R}^n$.	

\vspace{.1in}

It follows from property (P2) that any Banach function space $X$ is complete (see \cite{Lux}). We also have that, for measurable functions $f_1$ and $f_2$, if $|f_1|\leq |f_2|$ a.e. in $\mathbb{R}^n$ and $f_2 \in X$, then it follows that $f_1\in X$ and $\|f_1\|_X \leq \|f_2\|_X$.

Given  a Banach function space $X$, the K\"othe dual space (or the associate space) to $X$, denoted by $X'$, is the set of all measurable functions $f$ such that $f g \in L^1(\mathbb{R}^n)$ for all $g \in X$. It turns out that $X'$ is also a Banach function space with the associate metric function $\rho'(f)$, $f\geq 0$, defined by
$$\rho'(f):= \sup\left\{\int |fg| dx: g\in X, \, \|g\|_{X}\leq 1  \right\}.$$ 

By definition, the second associate space $X''$ to $X$ is given by $X''=(X')'$, i.e., $X''$ is the K\"othe dual space to $X'$.  The following theorems are fundamental in the theory of Banach function spaces (see \cite{Lux}).

\begin{theorem}\label{XX''} Every Banach function space $X$ coincides
	with its second associate space $X''$, i.e., $X=X''$ with  equality of norms.	
\end{theorem}

\begin{theorem}\label{X*X'} $X^*=X'$ (isometrically) if and only if the space $X$ has an absolutely continuous norm. 
\end{theorem} 

Here we say that $X$ has an absolutely continuous norm if the following properties are satisfied for any $f\in X$:

(a) If $E$ is a bounded set of $\mathbb{R}^n$ and $E_j$ are measurable subsets of $E$ such that $|E_j|\rightarrow 0$ as $j\rightarrow \infty$, the $\|f\chi_{E_j}\|_{X}\rightarrow0$ as $j\rightarrow \infty$.

(b) $ \|f \chi_{\mathbb{R}^n\setminus B_j(0)} \|_{X} \rightarrow 0$ as $j\rightarrow \infty$.

It is known that  $X$ has an absolutely continuous norm
if  and only if any  sequence $f_j\in X$ such that $|f_j| \downarrow 0$
a.e. in $\mathbb{R}^n$ has the property that $\| f_j\|_X \downarrow 0$ (see \cite[page 14]{Lux}).

\section{Proof of Theorem \ref{Main1}}
We will need the so-called ``integrating by parts" lemma that was obtained in \cite[Lemma 2.1]{Ver1}. 
\begin{lemma}\label{IBPL} Let $t\geq 1$ and suppose that $f$ is a nonnegative measurable function in $\mathbb{R}^n$. Then it holds that
$$(I_\alpha f)^t \leq A I_\alpha[ f(I_\alpha f)^{t-1}]$$
everywhere in $\mathbb{R}^n$.
\end{lemma}

\begin{proof}[Proof of Theorem \ref{Main1}]
Let  $f\in L^\infty(\mathbb{R}^n)$, $f\geq 0$, with compact support, or $f$ be a  nonnegative measurable function such that $I_\alpha f \in L^q( {\rm cap}_{\alpha, s})$, $q\geq 1$. Note that the former is a special case of the later.
To prove \eqref{capstrong2}, we may assume that $\int_{\mathbb{R}^n} f^s (I_\alpha f)^{q-s} dx<+\infty$.  Then under the condition $I_\alpha f \in L^q( {\rm cap}_{\alpha, s})$ we can check that $\|f (I_\alpha f)^{q-1}\|_{(M^{\alpha, s}_{s',s})'}<+\infty$. Indeed, for any $0\leq g\in \dot{M}^{\alpha, s}_{s',s}$, recall from \eqref{tr2} that 
$$\|g\|_{\dot{M}^{\alpha,s}_{s',s}}\simeq \sup_{K}\left(\frac{\int_{K}g(x)^{s'}dx}{\text{cap}_{\alpha,s}(K)}\right)^{1/s'},$$
where the supremum is taken over all compact sets $K\subset{\mathbb{R}}^{n}$ such that $\text{cap}_{\alpha,s}(K)\not=0$.
Thus, by H\"older's inequality,
\begin{align*}
\int_{\mathbb{R}^n}f (I_\alpha f)^{q-1} g dx &\leq \left(\int_{\mathbb{R}^n}f^s (I_\alpha f)^{q-s}  dx\right)^{\frac{1}{s}} \left(\int_{\mathbb{R}^n}g^{s'} (I_\alpha f)^{q}  dx\right)^{\frac{1}{s'}}\\
&\leq \left(\int_{\mathbb{R}^n}f^s (I_\alpha f)^{q-s}  dx\right)^{\frac{1}{s}} \left(\int_{0}^\infty \left(\int_{\{(I_\alpha f)^{q}>t\}} g^{s'} dy \right)dt  \right)^{\frac{1}{s'}}\\
&\lesssim  \left(\int_{\mathbb{R}^n}f^s (I_\alpha f)^{q-s}  dx\right)^{\frac{1}{s}} \|g\|_{M^{\alpha,s}_{s',s}} \left(\int_{\mathbb{R}^n} (I_\alpha f)^{q}  d {\rm cap}_{\alpha,s}\right)^{\frac{1}{s'}}.
\end{align*}	
This shows that 
\begin{align}\label{fIfq}
\|f (I_\alpha f)^{q-1}\|_{(\dot{M}^{\alpha, s}_{s',s})'} &\lesssim  \left(\int_{\mathbb{R}^n}f^s (I_\alpha f)^{q-s}  dx\right)^{\frac{1}{s}}  \left(\int_{\mathbb{R}^n} (I_\alpha f)^{q}  d {\rm cap}_{\alpha,s}\right)^{\frac{1}{s'}}\\
&<+\infty,\nonumber
\end{align}
as claimed.

Also, recall that the space of quasi-continuous functions $g$ such that  $\|g\|_{L^1({\rm cap}_{\alpha,s})}<+\infty$ is normable and its dual space consists of locally finite signed measures $\mu$ such that   
$$
\|\mu\|_{\dot{\mathfrak{M}}^{\alpha,s}}:= \sup_{K}\frac{|\mu|(K)}{{\rm cap}_{\alpha,s}(K)}<+\infty;
$$
see, e.g., \cite{OP1}.
Suppose for now that $f$ is a bounded function with compact support. Then by Lemma \ref{IBPL} and Hahn-Banach theorem,  we find
\begin{align}\label{Lqcap}
\int_{\mathbb{R}^n} (I_\alpha f)^q d {\rm cap}_{\alpha,s} &\leq A  \int I_\alpha[ f(I_\alpha f)^{q-1}] d {\rm cap}_{\alpha,s}\\
&\leq A \sup_{\|\mu\|_{\dot{\mathfrak{M}}^{\alpha,s}}\leq 1} \int I_\alpha[ f(I_\alpha f)^{q-1}] d|\mu|\nonumber\\
&= A \sup_{\|\mu\|_{\dot{\mathfrak{M}}^{\alpha,s}}\leq 1} \int (I_\alpha |\mu|) f (I_\alpha f)^{q-1} dx \nonumber\\
&\leq A \|f (I_\alpha f)^{q-1}\|_{(\dot{M}^{\alpha, s}_{s',s})'}  \sup_{\|\mu\|_{\dot{\mathfrak{M}}^{\alpha,s}}\leq 1}  \|I_\alpha |\mu|\|_{\dot{M}^{\alpha, s}_{s', s}} \nonumber\\
&\leq A \|f (I_\alpha f)^{q-1}\|_{(\dot{M}^{\alpha, s}_{s',s})'}.\nonumber
\end{align}
Note that in the last inequality above we used \cite[Theorem 2.1]{MV}. Now for any nonnegative measurable function $f$ such that $I_\alpha f \in L^q( {\rm cap}_{\alpha, s})$, we let 
\begin{equation}\label{FN}
f_N= f\cdot \chi_{\{|f|\leq N\}} \cdot \chi_{\{|x|\leq N\}}, \quad N=1,2, \dots
\end{equation}
Then using \eqref{Lqcap} and $q\geq 1$ we have 
$$\int_{\mathbb{R}^n} (I_\alpha f_N)^q d {\rm cap}_{\alpha,s} \leq A \|f_N (I_\alpha f_N)^{q-1}\|_{(\dot{M}^{\alpha, s}_{s', s})'}\leq A \|f (I_\alpha f)^{q-1}\|_{(\dot{M}^{\alpha, s}_{s', s})'},$$ 
which yields
\begin{equation}\label{Lqcap2}
\int_{\mathbb{R}^n} (I_\alpha f)^q d {\rm cap}_{\alpha,s} \leq A \|f (I_\alpha f)^{q-1}\|_{(\dot{M}^{\alpha, s}_{s',s})'}.
\end{equation}

At this point, combining  \eqref{fIfq} and \eqref{Lqcap2} we obtain
\begin{equation}\label{mvn}
\|f (I_\alpha f)^{q-1} \|_{(\dot{M}^{\alpha, s}_{s', s})'}\leq A \int_{\mathbb{R}^n} f^s (I_\alpha f)^{q-s}  dx.
\end{equation}

Inequality \eqref{capstrong2} now follows from   \eqref{Lqcap2} and \eqref{mvn}.

%
%
%
\end{proof}

\begin{remark} The proof above 	actually shows that inequality \eqref{capstrong2} holds under the assumption $f (I_\alpha f)^{q-1}\in (\dot{M}^{\alpha, s}_{s', s})'$.
Note also that by 	\eqref{Lqcap2} we have $\int_{\mathbb{R}^n} I_\alpha g d {\rm cap}_{\alpha,s} \leq A \|g \|_{(\dot{M}^{\alpha, s}_{s',s})'}$, $g\geq 0$, and thus by 
\eqref{mvn}, it holds that
\begin{equation*}
\int_{\mathbb{R}^n} I_\alpha[f (I_\alpha f)^{q-1}] d{\rm cap}_{\alpha,s}\leq C \int_{\mathbb{R}^n} f^s (I_\alpha f)^{q-s}  dx
\end{equation*}
provided $f (I_\alpha f)^{q-1}\in (\dot{M}^{\alpha, s}_{s', s})'$, $f\geq 0$. By approximation, the condition $f (I_\alpha f)^{q-1}\in (\dot{M}^{\alpha, s}_{s', s})'$ is not needed in the case $q\geq s$.

On the other hand,	for the case $q>s$, we notice that Theorem  \ref{Main1} can be proved simply by using the equivalence (see Theorem \ref{Fnorm})
	$$\|u\|_{L^q({\rm cap}_{\alpha,s})}^q \simeq\inf\left\{ \int_{\mathbb{R}^n} g^s dx: g\geq 0, I_\alpha g\geq |u|^{\frac{q}{s}} ~ {\rm q.e.}\right\}.$$
	Indeed, then for $u=I_\alpha f$, we have $u^\frac{q}{s}\leq c I_\alpha(f (I_\alpha f)^{\frac{q}{s}-1})$, and the result immediately follows. 
\end{remark}

\section{Proof of Theorem \ref{Main2}}
\begin{proof}[Proof of Theorem \ref{Main2}]
Suppose that $1\leq q<s$. We are to show that there is some constant $C>0$ such that
\begin{align}\label{w-q}
\left(\int_{\mathbb{R}^{n}}(I_{\alpha}f)^{q}d{\rm cap}_{\alpha,s}\right)^{\frac{1}{q}}\leq C\left(\int_{\mathbb{R}^{n}}f^{s} w ^{q-s}dx\right)^{\frac{1}{s}}
\end{align}
for any weight $w$ with $\|w\|_{L^{q}({\rm cap}_{\alpha,s})}\leq 1$. Observe that  by \eqref{Lqcap2} we have
\begin{equation*}
\int_{\mathbb{R}^n} I_\alpha f d {\rm cap}_{\alpha,s} \leq A \|f \|_{(\dot{M}^{\alpha, s}_{s',s})'}.
\end{equation*}
Thus, by a characterization of the space $(\dot{M}^{\alpha,s}_{s',s})'$ in \cite{OP1} (or by  Theorems \ref{Fnorm} and \ref{barNMprime} below) we see that 
\eqref{w-q} holds in the case $q=1$. Note also that the case $q=s$ is just \eqref{CSIM}.

Suppose now that $1<q<s$. For a weight $w$ with $\|w\|_{L^{q}({\rm cap}_{\alpha,s})}\leq 1$, assume without loss of generality that 
\begin{align*}
\int_{\mathbb{R}^{n}}f^{s} w ^{q-s}dx<+\infty.
\end{align*}
 We express $f$ as 
\begin{align*}
f=a^{1-\theta}\cdot b^{\theta},
\end{align*}
where $\theta\in(0,1)$ is such that 
\begin{align*}
\frac{1}{q}=\frac{1-\theta}{s}+\frac{\theta}{1},
\end{align*}
and 
\begin{align*}
a=f\cdot w^{q(1-s)\frac{\theta}{s}},\qquad b=f\cdot w^{q(1-s)\frac{\theta-1}{s}}.
\end{align*}
Note that 
\begin{align*}
\left\| w^{q(1-s)\frac{1}{1-s}}\right\|_{L^{1}({\rm cap}_{\alpha,s})}=\int_{\mathbb{R}^{n}} w^{q}d{\rm cap}_{\alpha,s}\leq 1,
\end{align*}
and we have by H\"older's inequality that 
\begin{align*}
I_{\alpha}f=I_{\alpha}(a^{1-\theta}\cdot b^{\theta})\leq(I_{\alpha}a)^{1-\theta}(I_{\alpha}b)^{\theta}.
\end{align*}
On the other hand, we see that 
\begin{align*}
(1-\theta)\frac{q}{s}+\theta q=q\left[\frac{\theta}{1}+\frac{1}{s}(1-\theta)\right]=q\cdot\frac{1}{q}=1.
\end{align*}
As a consequence, applying H\"older's inequality again, we obtain
\begin{align*}
\int_{\mathbb{R}^{n}}(I_{\alpha}f)^{q}d{\rm cap}_{\alpha,s}&\leq\int_{\mathbb{R}^{n}}(I_{\alpha}a)^{(1-\theta)q}(I_{\alpha}b)^{\theta q}d{\rm cap}_{\alpha,s}\\
&\leq C\left(\int_{\mathbb{R}^{n}}(I_{\alpha}a)^{s}d{\rm cap}_{\alpha,s}\right)^{(1-\theta)\frac{q}{s}}\left(\int_{\mathbb{R}^{n}}(I_{\alpha}b)d{\rm cap}_{\alpha,s}\right)^{\theta q}.
\end{align*}
Therefore,
\begin{align*}
\|I_{\alpha}f\|_{L^{q}({\rm cap}_{\alpha,s})}&\leq C\|I_{\alpha}a\|_{L^{s}({\rm cap}_{\alpha,s})}^{1-\theta}\|I_{\alpha}b\|_{L^{1}({\rm cap}_{\alpha,s})}^{\theta}\\
&\leq C\|a\|_{L^{s}(\mathbb{R}^{n})}^{1-\theta}\left(\int_{\mathbb{R}^{n}}b^{s}(w^{q(1-s)})^{\frac{1}{1-s}\cdot(1-s)}dx\right)^{\frac{\theta}{s}}\\
&\leq C\left(\int_{\mathbb{R}^{n}}f^{s}(w^{q(1-s)})^{\theta}dx\right)^{\frac{1-\theta}{s}}\left(\int_{\mathbb{R}^{n}}f^{s}(w^{q(1-s)})^{\theta}dx\right)^{\frac{\theta}{s}}\\
&=C\left(\int_{\mathbb{R}^{n}}f^{s}w^{q-s}dx\right)^{\frac{1}{s}}
\end{align*}
by noting that $q(1-s)\theta=q-s$. The proof is now complete.
\end{proof}

\begin{remark}
	\rm Suppose that $I_{\alpha}f\in L^{q}({\rm cap}_{\alpha,s})$ and $f\geq 0$. Let 
	\begin{align*}
	w=\frac{I_{\alpha}f}{\|I_{\alpha}f\|_{L^{q}({\rm cap}_{\alpha,s})}}.
	\end{align*}
	Then $\|w\|_{L^{q}({\rm cap}_{\alpha,s})}=1$, and hence 
	\begin{align*}
	\|I_{\alpha}f\|_{L^{q}({\rm cap}_{\alpha,s})}&\leq C\left(\int_{\mathbb{R}^{n}}f^{s}\left(\frac{I_{\alpha}f}{\|I_{\alpha}f\|_{L^{q}({\rm cap}_{\alpha,s})}}\right)^{q-s}dx\right)^{\frac{1}{s}}\\
	&=C \frac{1}{\|I_{\alpha}f\|_{L^{q}({\rm cap}_{\alpha,s})}^{\frac{q}{s}-1}}\left(\int_{\mathbb{R}^{n}}f^{s}(I_{\alpha}f)^{q-s}dx\right)^{\frac{1}{s}}.
	\end{align*}
	Routine simplification gives 
	\begin{align*}
	\|I_{\alpha}f\|_{L^{q}({\rm cap}_{\alpha,s})}\leq C\left(\int_{\mathbb{R}^{n}}f^{s}(I_{\alpha}f)^{q-s}dx\right)^{\frac{1}{q}},
	\end{align*}
	which provides an alternative proof of Theorem \ref{Main1} for the case $1\leq q<s$.
\end{remark}

\section{Proof of Theorem \ref{KV-q}}
\begin{proof}[Proof of Theorem \ref{KV-q}]
Recall that for $1\leq q<s$, $\dot{KV}_q=\dot{KV}_q^{\alpha,s}$ is  defined as the set of all Lebesgue measurable functions $f$ in $\mathbb{R}^n$ such that    $\|f\|_{\dot{KV}_q}<+\infty$,
where 
\begin{align*}
\|f\|_{\dot{KV}_q}:=\inf\left\{\left(\int_{\mathbb{R}^{n}}h^{s}(I_{\alpha}h)^{q-s}dx\right)^{\frac{1}{q}}:h\in \dot{\widetilde{O}}_{s,q}^{\alpha,s},~h\geq|f|~{\rm a.e.}\right\}.
\end{align*}	
Note that the condition  $h\in \dot{\widetilde{O}}_{q}^{\alpha,s}$ is necessary, or else putting $h=e^{\left|\cdot\right|^{2}}\geq\chi_{B_{1}(0)}$ will result in an absurd definition.
We are to show that 
\begin{align*}
\dot{KV}_q \approx \dot{\widetilde{O}}_{q}^{\alpha,s}.
\end{align*}

The direction $\dot{KV}_q\hookrightarrow\widetilde{O}_{q}^{\alpha,s}$ is rather easy. 
Indeed, for $f\in \dot{KV}_q$ and $h\in \dot{\widetilde{O}}_{q}^{\alpha,s}$ such that $h\geq |f|$ a.e., it holds that $\|I_{\alpha}h\|_{L^{q}({\rm cap}_{\alpha,s})}<+\infty$ and
\begin{align*}
\|I_{\alpha}h\|_{L^{q}({\rm cap}_{\alpha,s})}\leq c\left(\int_{\mathbb{R}^{n}}h^{s}(I_{\alpha}h)^{q-s}dx\right)^{\frac{1}{q}}
\end{align*}
by Theorems \ref{Main1} and \ref{Main2}. Then
\begin{align*}
\|f\|_{\dot{\widetilde{O}}_{q}^{\alpha,s}}\leq \|h\|_{\dot{\widetilde{O}}_{q}^{\alpha,s}}&\leq c\left(\int_{\mathbb{R}^{n}}h^{s}\left(\frac{I_{\alpha}h}{\|I_{\alpha}h\|_{L^{q}({\rm cap}_{\alpha,s})}}\right)^{q-s}dx\right)^{\frac{1}{s}}\\
&=c\cdot\|I_{\alpha}h\|_{L^{q}({\rm cap}_{\alpha,s})}^{\frac{s-q}{s}}\left(\int_{\mathbb{R}^{n}}h^{s}(I_{\alpha}h)^{q-s}dx\right)^{\frac{1}{s}}\\
&\leq c\left(\int_{\mathbb{R}^{n}}h^{s}(I_{\alpha}h)^{q-s}dx\right)^{\frac{s-q}{sq}+\frac{1}{s}}\\
&=c\left(\int_{\mathbb{R}^{n}}h^{s}(I_{\alpha}h)^{q-s}dx\right)^{\frac{1}{q}}.
\end{align*}
Taking infimum with respect to all such $h$, one obtains the embedding $\dot{KV}_q\hookrightarrow \dot{\widetilde{O}}_{q}^{\alpha,s}$.

On the other hand, the proof of the  embedding $\dot{\widetilde{O}}_{q}^{\alpha,s}\hookrightarrow \dot{KV}_q$ is more technical. To this end, let $f\in \dot{\widetilde{O}}^{\alpha,s}_{q}$ be given. Choose some weight $w \in L^{q}({\rm cap}_{\alpha,s})$ with $\|w\|_{L^{q}({\rm cap}_{\alpha,s}}\leq 1$ and 
\begin{align*}
\left(\int_{\mathbb{R}^{n}}|f|^{s}w^{q-s}dx\right)^{\frac{1}{s}}\leq 2\cdot\|f\|_{\dot{\widetilde{O}}_{q}^{\alpha,s}}.
\end{align*}
By Theorem \ref{Fnorm} and Lemma \ref{extreme} below, there is a function $0\leq \varphi\in L^s(\mathbb{R}^n)$ such that $I_{\alpha}\varphi\geq w^{\frac{q}{s}}$ q.e. and 
\begin{align*}
\|\varphi\|_{L^{s}(\mathbb{R}^{n})}^{\frac{s}{q}} \simeq \|w\|_{L^q({\rm cap}_{\alpha,s})}.
\end{align*}
Consider $h=|f|\cdot w^{\frac{q}{s}-1}+\delta\cdot\varphi$, where $\delta>0$ will be determined later. Since $h\in L^{s}(\mathbb{R}^{n})$, we have $(I_\alpha h)^{\frac{s}{q}}\in L^q({\rm cap}_{\alpha,s})$ by \eqref{CSIM} and moreover,
$$\int_{\mathbb{R}^n} \left[h\cdot(I_{\alpha}h)^{\frac{s}{q}-1}\right]^s \left[(I_\alpha h)^{\frac{s}{q}}\right]^{q-s} dx=\int_{\mathbb{R}^n}  h^s dx<+\infty.$$
This shows that 
\begin{equation}\label{hIhO}
h\cdot(I_{\alpha}h)^{\frac{s}{q}-1}\in \dot{\widetilde{O}}_{q}^{\alpha,s},
\end{equation}
 and hence by the definition of $\dot{KV}_q$ and Lemma \ref{IBPL},
\begin{align*}
\left\|h\cdot(I_{\alpha}h)^{\frac{s}{q}-1}\right\|_{\dot{KV}_q}&\leq\left(\int_{\mathbb{R}^{n}}\left(h\cdot(I_{\alpha}h)^{\frac{s}{q}-1}\right)^{s}\left(I_{\alpha}\left(h\cdot(I_{\alpha}h)^{\frac{s}{q}-1}\right)\right)^{q-s}dx\right)^{\frac{1}{q}}\\
&\leq c\left(\int_{\mathbb{R}^{n}}h^{s}\cdot(I_{\alpha}h)^{\frac{s^{2}}{q}-s}(I_{\alpha}h)^{\frac{s}{q}\cdot(q-s)}dx\right)^{\frac{1}{q}}\\
&=c\cdot\|h\|_{L^{s}(\mathbb{R}^{n})}^{\frac{s}{q}}.
\end{align*} 
On the other hand, since $1\leq q<s$, we have 
\begin{align*}
h\cdot(I_{\alpha}h)^{\frac{s}{q}-1}&\geq |f|\cdot w^{\frac{q}{s}-1}\cdot(I_{\alpha} (\delta \cdot \varphi))^{\frac{s}{q}-1}\\
&\geq\delta^{\frac{s}{q}-1} |f|\cdot w^{\frac{q}{s}-1}(I_{\alpha}\varphi)^{\frac{s}{q}-1}\\
&\geq\delta^{\frac{s}{q}-1} |f|\cdot w^{\frac{q}{s}-1} w^{\frac{q}{s}(\frac{s}{q}-1)}\\
&=\delta^{\frac{s}{q}-1}|f|.
\end{align*}
Hence,
\begin{align*}
\|f\|_{\dot{KV}_q}&\leq\delta^{1-\frac{s}{q}}\left\|h\cdot(I_{\alpha}h)^{\frac{s}{q}-1}\right\|_{\dot{KV}_q}\\
&\leq c\cdot\delta^{1-\frac{s}{q}}\cdot\|h\|_{L^{s}(\mathbb{R}^{n})}^{\frac{s}{q}}\\
&\leq c\cdot\delta^{1-\frac{s}{q}}\left[\left(\int_{\mathbb{R}^{n}}|f|^{s} w^{q-s}dx\right)^{\frac{1}{s}\cdot\frac{s}{q}}+\delta^{\frac{s}{q}}\cdot\|\varphi\|_{L^{s}(\mathbb{R}^{n})}^{\frac{s}{q}}\right]\\
&\leq c\cdot\delta^{1-\frac{s}{q}}\left(\|f\|_{\dot{\widetilde{O}}_{q}^{\alpha,s}}^{\frac{s}{q}}+\delta^{\frac{s}{q}}\right).
\end{align*}
Now we simply let $\delta=\|f\|_{\dot{\widetilde{O}}_{q}^{\alpha,s}}$, then the embedding $\dot{\widetilde{O}}_{q}^{\alpha,s}\hookrightarrow \dot{KV}_q$ holds.
\end{proof}

\section{The spaces $\dot{M}^{\alpha,s}_{p, r}$ and $(\dot{M}^{\alpha,s}_{p, r})'$}

It is easy to see from the definition that the space $\dot{M}^{\alpha,s}_{p, r}$, $s>1$, $0<\alpha< n/s$, $p>1$, and $0<r\leq s$, is a Banach function space in $\mathbb{R}^n$ and so is its K\"othe dual space $(\dot{M}^{\alpha,s}_{p, r})'$. Moreover, we have the following results in which the case $r=s$ was considered earlier in \cite{OP1}.

\begin{proposition}\label{M'*} Let $s>1$, $0<\alpha< n/s$, $p>1$, and $0<r\leq s$. The space $\dot{M}^{\alpha,s}_{p, r}$ is $p$-convex and thus $(\dot{M}^{\alpha,s}_{p, r})'$ is $p'$-concave. Then 
	\begin{align*}
	[(\dot{M}^{\alpha,s}_{p, r})']^*=\dot{M}^{\alpha,s}_{p,r}=(\dot{M}^{\alpha,s}_{p,r})''
	\end{align*}
	holds isometrically. Moreover, $(\dot{M}^{\alpha,s}_{p,r})'$ has an absolutely continuous norm and  is a separable Banach  space.
\end{proposition}

\begin{proof}
	We first show that $\dot{M}_{p,r}^{\alpha,s}$ is $p$-convex. Let $\{f_{i}\}_{i=1}^{m}\subseteq \dot{M}_{p,r}^{\alpha,s}$, $m\in\mathbb{N}$. Then
	\begin{align*}
	\left(\int_{\mathbb{R}^{n}}(I_{\alpha}h)^{r}\left(\left(\sum_{i=1}^{m}|f_{i}|^{p}\right)^{\frac{1}{p}}\right)^{p}dx\right)^{\frac{1}{p}}&=\left(\sum_{i=1}^{m}\int_{\mathbb{R}^{n}}(I_{\alpha}h)^{r}|f_{i}|^{p}dx\right)^{\frac{1}{p}}\\
	&\leq\left(\sum_{i=1}^{m}\|f_{i}\|_{\dot{M}_{p,r}^{\alpha,s}}^{p}\cdot\|h\|_{L^{s}(\mathbb{R}^{n})}^{r}\right)^{\frac{1}{p}}\\
	&=\left(\sum_{i=1}^{m}\|f_{i}\|_{\dot{M}_{p,r}^{\alpha,s}}^{p}\right)^{\frac{1}{p}}\|h\|_{L^{s}(\mathbb{R}^{n})}^{\frac{r}{p}}
	\end{align*}
	holds for all nonnegative $h\in L^{s}(\mathbb{R}^{n})$. As a result, we have 
	\begin{align*}
	\left\|\left(\sum_{i=1}^{m}|f_{i}|^{p}\right)^{\frac{1}{p}}\right\|_{\dot{M}_{p,r}^{\alpha,s}}\leq\left(\sum_{i=1}^{m}\|f_{i}\|_{\dot{M}_{p,r}^{\alpha,s}}^{p}\right)^{\frac{1}{p}},
	\end{align*}
	which justifies the $p$-convexity of $\dot{M}_{p,r}^{\alpha,s}$ with $p$-convexity constant $1$. 
	
	Moreover, for any $\{g_{i}\}_{i=1}^{m}\subseteq(\dot{M}_{p,r}^{\alpha,s})'$, $m\in\mathbb{N}$, since $\ell^{p'}((\dot{M}_{p,r}^{\alpha,s})^{\ast})=[\ell^{p}(\dot{M}_{p,r}^{\alpha,s})]^{\ast}$, we have 
	\begin{align*}
	&\left(\sum_{i=1}^{m}\|g_{i}\|_{(\dot{M}_{p,r}^{\alpha,s})'}^{p'}\right)^{\frac{1}{p'}}=\left(\sum_{i=1}^{m}\|g_{i}\|_{(\dot{M}_{p,r}^{\alpha,s})^{\ast}}^{p'}\right)^{\frac{1}{p'}}\\
	&=\sup_{\|\{f_{i}\}_{i=1}^{m}\|_{\ell^{p}(\dot{M}_{p,r}^{\alpha,s})}\leq 1}\sum_{i=1}^{m}\int_{\mathbb{R}^{n}}f_{i}g_{i}dx\\
	&\leq\sup_{\|\{f_{i}\}_{i=1}^{m}\|_{\ell^{p}(\dot{M}_{p,r}^{\alpha,s})}\leq 1}\int_{\mathbb{R}^{n}}\left(\sum_{i=1}^{m}|f_{i}|^{p}\right)^{\frac{1}{p}}\left(\sum_{i=1}^{m}|g_{i}|^{p'}\right)^{\frac{1}{p'}}dx\\
	&=\sup_{\|\{f_{i}\}_{i=1}^{m}\|_{\ell^{p}(\dot{M}_{p,r}^{\alpha,s})}\leq 1}\left\|\left(\sum_{i=1}^{m}|f_{i}|^{p}\right)^{\frac{1}{p}}\right\|_{\dot{M}_{p,r}^{\alpha,s}}\left\|\left(\sum_{i=1}^{m}|g_{i}|^{p'}\right)^{\frac{1}{p'}}\right\|_{(\dot{M}_{p,r}^{\alpha,s})'}\\
	&\leq\sup_{\|\{f_{i}\}_{i=1}^{m}\|_{\ell^{p}(\dot{M}_{p,r}^{\alpha,s})}\leq 1}\left(\sum_{i=1}^{m}\|f_{i}\|_{\dot{M}_{p,r}^{\alpha,s}}^{p}\right)^{\frac{1}{p}}\left\|\left(\sum_{i=1}^{m}|g_{i}|^{p'}\right)^{\frac{1}{p'}}\right\|_{(\dot{M}_{p,r}^{\alpha,s})'}.
	\end{align*}
	We conclude that $(\dot{M}_{p,r}^{\alpha,s})'$ is $p'$-concave with $p'$-concavity constant $1$, i.e.,
	\begin{align*}
	\left(\sum_{i=1}^{m}\|g_{i}\|_{(\dot{M}_{p,r}^{\alpha,s})'}^{p'}\right)^{\frac{1}{p'}}\leq\left\|\left(\sum_{i=1}^{m}|g_{i}|^{p'}\right)^{\frac{1}{p'}}\right\|_{(\dot{M}_{p,r}^{\alpha,s})'}.
	\end{align*}
	Hence the space $(\dot{M}_{p,r}^{\alpha,s})'$ has an absolutely continuous norm (see \cite[Proposition 1.a.7]{LT}) and by Theorems \ref{XX''} and \ref{X*X'}, we have
	\begin{align*}
	[(\dot{M}^{\alpha,s}_{p, r})']^*=\dot{M}^{\alpha,s}_{p,r}=(\dot{M}^{\alpha,s}_{p,r})''.
	\end{align*}
	Moreover, $(\dot{M}_{p,r}^{\alpha,s})'$ is a separable Banach space (see \cite{Lux}).
\end{proof}

The following theorem provides an explicit norm for the space $L^q({\rm cap}_{\alpha,s})$ when $q\geq 1$.

\begin{theorem}\label{Fnorm} Suppose that $s>1$, $0<\alpha<\frac{n}{s}$, and $0<r\leq s$. Let $\dot{F}=\dot{F}^{\alpha,s}_r$ be the space of ${\rm cap}_{\alpha,s}$ q.e. defined functions $u:\mathbb{R}^n\rightarrow [-\infty,\infty]$ such that 
	$\|u\|_{\dot{F}}<+\infty$, where
	$$\|u\|_{\dot{F}}:=\inf\{ \|f\|_{L^s(\mathbb{R}^{n})}^r: f\geq 0, \, I_{\alpha} f\geq |u|^{\frac1r} {\rm ~q.e.}\}.$$	
	The convention here is that  $\inf \emptyset :=+\infty$. Then $\left\|\cdot\right\|_{\dot{F}}$ is a norm for $\dot{F}$ and the normed space $(\dot{F}, \left\|\cdot\right\|_{\dot{F}})$ is isomorphic to  $L^{\frac{s}{r}}({\rm cap}_{\alpha,s})$.
\end{theorem}	

\begin{proof} Let $B$ be the set of all ${\rm cap}_{\alpha,s}$ q.e. defined functions $u:\mathbb{R}^n\rightarrow [-\infty,\infty]$ such that there exists an $f\in L^s(\mathbb{R}^n)$, $f\geq 0$, $\|f\|_{L^s} < 1$, such that $I_{\alpha} f\geq |u|^{\frac1r}$ q.e. We claim that $B$ is convex. Indeed, let $u_1, u_2\in B$ and $f_1, f_2\geq 0$ be such that $\|f_i\|_{L^s}<1$ and 
	$(I_\alpha  f_i)^r\geq |u_i|$ q.e. for all $i=1,2$. For any $\lambda\in (0,1)$, we set $g=\left(\lambda f_1^s +(1-\lambda) f_2^s\right)^{1/s}$. Then 
	$$\|g\|_{L^s(\mathbb{R}^{n})}^s=\lambda\|f_1\|_{L^s(\mathbb{R}^{n})}^s + (1-\lambda)\|f_2\|_{L^s(\mathbb{R}^{n})}^s<1.$$
	
	Moreover, by reverse Minkowski inequality, 
	\begin{align*}
	(I_\alpha  g)^s &=\gamma(n,\alpha) \left[\int_{\mathbb{R}^n} \left(\lambda f_1(y)^s +(1-\lambda) f_2(y)^s\right)^{1/s} |\cdot-y|^{\alpha-n} dy\right]^{s}\\
	&\geq \gamma(n,\alpha) \left[\int_{\mathbb{R}^n} \lambda^{1/s} f_1(y)  |\cdot-y|^{\alpha-n} dy\right]^{s} \\
	&\quad + \gamma(n,\alpha) \left[\int_{\mathbb{R}^n} (1-\lambda)^{1/s} f_2(y)  |x-y|^{\alpha-n} dy\right]^{s}\\
	&= \lambda (I_\alpha  f_1)^s +  (1-\lambda) (I_\alpha  f_2)^s.
	\end{align*}
	
	Thus, as $s/r\geq 1$, 
	\begin{align*}
	(I_\alpha  g)^s &\geq  \lambda |u_1|^{s/r} +  (1-\lambda) |u_2|^{s/r}\geq (\lambda |u_1| +  (1-\lambda) |u_2|)^{\frac{s}{r}}\quad {\rm q.e.}
	\end{align*}
	This gives $\lambda u_1 +  (1-\lambda) u_2 \in B$, and thus $B$ is convex as claimed.
	
	Next, we observe that 
	\begin{equation}\label{min-norm}
	\|u\|_{\dot{F}}=\inf\{ t>0: |u|/t\in B\} \qquad \forall u\in \dot{F}.
	\end{equation}
	Indeed, if $u\in \dot{F}$ and $t>0$ is such that $|u|/t\in B$, then there is a nonnegative $f\in L^s$ such that $\|f\|_{L^s(\mathbb{R}^{n})}<1$ and $I_\alpha  f\geq (|u|/t)^{\frac{1}{r}}$ q.e. Thus,
	$I_\alpha  (t^{\frac1r}f)\geq |u|^{\frac{1}{r}}$ q.e. and hence,
	$$\|u\|_{\dot{F}}\leq \|t^{\frac1r} f\|_{L^s(\mathbb{R}^{n})}^r= t \|f\|_{L^s(\mathbb{R}^{n})}^r < t.$$ 
	This implies that $$\|u\|_{\dot{F}}\leq \inf\{ t>0: |u|/t\in B\}.$$
	Conversely, let $u\in \dot{F}$ and $f\in L^s(\mathbb{R}^{n})$, $f\geq 0$, be such that  	$I_\alpha f \geq |u|^{\frac{1}{r}}$ q.e. 
	For any $\epsilon>0$, set $g=f/(\|f\|_{L^s(\mathbb{R}^{n})}+\epsilon)$. We have $\|g\|_{L^s(\mathbb{R}^{n})}<1$ and moreover,
	\begin{align*}
	I_\alpha  g &= \frac{1}{\|f\|_{L^s(\mathbb{R}^{n})}+\epsilon}I_\alpha f\geq \frac{|u|^{1/r}}{\|f\|_{L^s(\mathbb{R}^{n})}+\epsilon}\\
	&= \left[\frac{|u|}{(\|f\|_{L^s(\mathbb{R}^{n})}+\epsilon)^r}\right]^{1/r} = \left[\frac{|u|}{t_0}\right]^{1/r} {\rm ~ q.e.},
	\end{align*}
	where $t_0=(\|f\|_{L^s(\mathbb{R}^{n})}+\epsilon)^r$. Thus $|u|/t_0\in B$ and we get 
	$$\inf\{ t>0: |u|/t\in B\} \leq t_0 = (\|f\|_{L^s(\mathbb{R}^{n})}+\epsilon)^r.$$
	Letting $\epsilon \rightarrow 0$, we find 
	$$\inf\{ t>0: |u|/t\in B\}  \leq \|f\|_{L^s(\mathbb{R}^{n})}^r,$$
	which yields 
	$$\inf\{ t>0: |u|/t\in B\}  \leq \|u\|_{\dot{F}}.$$
	
	As $B$ is convex, it can be seen from \eqref{min-norm} that $\left\|\cdot\right\|_{\dot{F}}$ is sub-additive on $\dot{F}$. That $\|t u\|_{\dot{F}} =|t| \|u\|_{\dot{F}}$ is clear for all $u\in \dot{F}$. Moreover, 
	by \cite[Proposition 2.3.8]{AH} we find that $\|u\|_{\dot{F}}=0$ implies that $u=0$ q.e. Thus, we conclude that $\left\|\cdot\right\|_{\dot{F}}$ is a norm for $\dot{F}$. 
	
	Our next task is to show that 
	\begin{equation}\label{LF}
	\|u\|_{L^{s/r}({\rm cap}_{\alpha,s})} \leq C \|u\|_{\dot{F}}
	\end{equation}
	for all $u\in \dot{F}$. To this end, let $u\in \dot{F}$ and $f\in L^s, f\geq 0$, be such that  $I_\alpha  f\geq |u|^{1/r}$ q.e. We have 
	\begin{align*}
	\|u\|_{L^{s/r}({\rm cap}_{\alpha, s})}^{s/r}&=\int_0^\infty {\rm cap}_{\alpha,s}(\{|u|^{s/r}>t\}) dt\\
	&\leq \int_0^\infty {\rm cap}_{\alpha,s}(\{(I_\alpha f)^s >t\}) dt\\
	&\leq C \|f\|_{L^s(\mathbb{R}^{n})}^s
	\end{align*}
	by the capacitary strong type inequality \eqref{CSIM}. Now taking the infimum over such $f$, we get the bound \eqref{LF}.
	
	Finally, we show the converse to  \eqref{LF}, i.e.,
	\begin{equation}\label{FL}
	\|u\|_{\dot{F}}\leq D \|u\|_{L^{s/r}({\rm cap}_{\alpha,s})} 
	\end{equation}
	for any $u\in L^{s/r}({\rm cap}_{\alpha,s})$. For such $u$, we write 
	$$|u|= \sum_{i\in\mathbb{Z}} |u|\chi_{E_i},$$
	where $E_i= \{2^i\leq |u| < 2^{i+1}\}$. For any $i\in\mathbb{Z}$ and $\epsilon>0$, there exists $f_i\in L^s(\mathbb{R}^{n})$, $f_i\geq0$,  such that $I_\alpha  f_i \geq (|u| \chi_{E_i})^{\frac1r}$ q.e. and 
	$$\|f_i\|_{L^s(\mathbb{R}^{n})}^s< \|u\chi_{E_i}\|_{\dot{F}}^{\frac{s}{r}}+ \epsilon \, 2^{-|i|}.$$
	Let $f(x)=\sup_{i\in\mathbb{Z}} f_i(x)$. We have $f^s \leq \sum_{i\in\mathbb{Z}} f_i^s$, and thus
	\begin{align*}
	\|f\|_{L^s}^s &\leq \int_{\mathbb{R}^n} \sum_{i\in\mathbb{Z}} f_i^s dx=\sum_{i\in\mathbb{Z}} \int_{\mathbb{R}^n} f_i^s dx\\
	&< \sum_{i\in\mathbb{Z}} \|u\chi_{E_i}\|_{\dot{F}}^{\frac{s}{r}} + \sum_{i\in\mathbb{Z}} \epsilon \, 2^{-|i|} = \sum_{i\in\mathbb{Z}} \|u\chi_{E_i}\|_{\dot{F}}^{\frac{s}{r}} + 3 \epsilon. 
	\end{align*}
	Moreover, $I_\alpha  f\geq |u|^{\frac 1 r}$ q.e. because for q.e. $x\in E_i$, we have
	$$I_\alpha  f(x) \geq I_\alpha  f_i(x) \geq \left[|u(x)|\chi_{E_i}(x) \right]^{\frac 1 r}= |u(x)|^{\frac 1 r}.$$
	These bounds imply that 
	\begin{align*}
	\|u\|_{\dot{F}}^{\frac{s}{r}} &\leq \|f\|_{L^s(\mathbb{R}^{n})}^s < \sum_{i\in\mathbb{Z}} \|u\chi_{E_i}\|_{\dot{F}}^{\frac{s}{r}} +3\epsilon\\
	&\leq \sum_{i\in\mathbb{Z}} \|2^{i+1}\chi_{E_i}\|_{\dot{F}}^{\frac{s}{r}} +3\epsilon\\
	&= \sum_{i\in\mathbb{Z}} 2^{(i+1)s/r}  {\rm cap}_{\alpha,s}(E_i)+3\epsilon\\
	&\leq C  \sum_{i\in\mathbb{Z}} \int_{2^{i-1}}^{2^i} t^{s/r}  {\rm cap}_{\alpha,s}(\{|u|>t\}) \frac{dt}{t} +3\epsilon\\
	&\leq C \|u\|_{L^{s/r}({\rm cap}_{\alpha,s})}^{s/r} + 3\epsilon.
	\end{align*}
	Now, letting $\epsilon\rightarrow 0$, we obtain \eqref{FL} as desired.
\end{proof}

\begin{lemma}\label{extreme} For any $u\in \dot{F}$, where $\dot{F}=\dot{F}^{\alpha, s}_r$ is as in Theorem \ref{Fnorm}, there exists a unique nonnegative $f\in L^s(\mathbb{R}^n), $ such that 
$I_\alpha f\geq u^{\frac1r}$ q.e. and 
$$\|u\|_{\dot{F}}= \|f\|_{L^s(\mathbb{R}^n)}^r.$$	
\end{lemma}
\begin{proof} For any $u\in \dot{F}$, the set 
$$\Omega_{u^{1/r}}:= \{ g\in L^s(\mathbb{R}^n): g\geq 0, I_\alpha g \geq |u|^{\frac1r} {\rm ~ q.e.}  \}$$	
is obviously a non-empty convex subset of $L^s(\mathbb{R}^n)$. Moreover, arguing as in the proof of \cite[Proposition 2.3.9]{AH}, one has that 
$\Omega_{u^{1/r}}$ is closed in $L^s(\mathbb{R}^n)$.  The lemma now follows from the fact that $L^s(\mathbb{R}^n)$ is uniformly convex.
\end{proof}

\begin{proposition}\label{the F norm}
	The space $\dot{F}$ in Theorem \ref{Fnorm} possesses the Fatou's property, i.e.,
	\begin{align*}
	\left\|\liminf_{j\rightarrow\infty}u_{j}\right\|_{\dot{F}}\leq\liminf_{j\rightarrow\infty}\|u_{j}\|_{\dot{F}}
	\end{align*}
	for any sequence $\{u_{j}\}\subset \dot{F}$ of nonnegative functions. 
\end{proposition}
\begin{proof} Note that if $u\leq v$ q.e. then $\|u\|_{\dot{F}} \leq \|v\|_{\dot{F}}$. Thus, it is enough to show the following monotone convergence property:
	\begin{align*}
	\left\|\lim_{j\rightarrow\infty}u_{j}\right\|_{\dot{F}} = \lim_{j\rightarrow\infty}\|u_{j}\|_{\dot{F}}
	\end{align*}
for any  non-decreasing	sequence $\{u_{j}\}\subset \dot{F}$ of nonnegative functions. The case $\lim_{j\rightarrow\infty}\|u_{j}\|_{\dot{F}}=\infty$ is obvious and thus we may assume that 
$$\lim_{j\rightarrow\infty}\|u_{j}\|_{\dot{F}}=L <+\infty.$$
Then for any $j\in \mathbb{N}$, by Lemma \ref{extreme}, there exists $f_j\geq 0$ such that  $I_{\alpha} f_j \geq |u_j|^{\frac1r}$ q.e. and
$$\|u_j\|_{\dot{F}}^{\frac1r} = \|f_j\|_{L^s(\mathbb{R}^{n})}.$$
Thus, for any $1\leq i<j $, we have $I_{\alpha} [\tfrac12(f_i+f_j)] \geq |u_i|^{\frac1r}$ q.e., and  so  
$$\|u_i\|_{\dot{F}}^{\frac1r} \leq  \left \|\tfrac12(f_i +f_j)\right\|_{L^s(\mathbb{R}^{n})}.$$
This enables us to apply \cite[Corollary 1.3.3]{AH} to the sequence $\{f_k/L^{1/r}\}_{k=1}^\infty$ to deduce that the sequence $\{f_j\}$ converges strongly to a nonnegative function $f$ in $L^s(\mathbb{R}^n)$ and 
$$\|f\|_{L^s(\mathbb{R}^n)} = \lim_{j\rightarrow\infty}  \|u_j\|_{\dot{F}}^{\frac1r}.$$ 
Then, by \cite[Proposition 2.3.8]{AH}, there exists a subsequence $\{f_{j_k}\}_{k=1}^\infty$ such that 
$I_\alpha f_{j_k} \rightarrow I_{\alpha} f$ q.e. as $k\rightarrow\infty$. This yields that $$I_{\alpha} f \geq \left|\lim_{j\rightarrow\infty} u_j\right|^{\frac1r} {\rm ~ q.e.}$$
and so $$\left\|\lim_{j\rightarrow\infty}u_{j}\right\|_{\dot{F}}\leq \|f\|_{L^s(\mathbb{R}^n)}^r=\lim_{j\rightarrow\infty}  \|u_j\|_{\dot{F}}.$$
As $\left\|\lim_{j\rightarrow\infty}u_{j}\right\|_{\dot{F}}\geq \lim_{j\rightarrow\infty}  \|u_j\|_{\dot{F}}$ is obvious, we obtain the proposition.
\end{proof}

We will need the following boundedness principle (or weak maximum principle) for Wolff  potentials that can be found, e.g.,  in \cite[Equ. (2.1)]{Ver2}. Recall that for 
$s>1$ and $0<\alpha<\frac{n}{s}$, the Wolff  potential  $W_{\alpha,s} \mu$ of a nonnegative measure $\mu$ is defined by 
$$W_{\alpha, s}\mu(x):= \int_0^\infty \left[\frac{\mu(B_t(x))}{t^{n-\alpha s}}\right]^{\frac{1}{s-1}} \frac{dt}{t}, \qquad x\in\mathbb{R}^n.$$

\begin{lemma}\label{bounded}
	Suppose that $s>1$, $0<\alpha<\frac{n}{s}$, and $\mu$ is a nonnegative measure on $\mathbb{R}^{n}$. Then
	\begin{align*}
	W_{\alpha,s}\mu(x)\leq 2^{\frac{n-\alpha s}{s-1}}\cdot\sup\left\{W_{\alpha,s}\mu(y):y\in{\rm supp}(\mu)\right\},\quad x\in\mathbb{R}^{n}.
	\end{align*}
\end{lemma}


The following characterizations of trace inequality in the so-called upper triangle case  play a key role in our study of $\dot{M}^{\alpha,s}_{p, r}$ and $(\dot{M}^{\alpha,s}_{p, r})'$ for $0<r<s$. 
We point out that the non-capacitary characterization (iii) in the next theorem was obtained in \cite{COV} and \cite{Ver1}. On the other hand, there is also a capacitary characterization obtained earlier in \cite{MN}, but that will not be used in the present paper.  

\begin{theorem}\label{upper-tri} Let $s>1$, $0<\alpha<\frac{n}{s}$, and $0<r<s$. Then the following statements are equivalent for a nonnegative locally finite (Borel) measure $\mu$ in $\mathbb{R}^n$.

	$\rm{(i)}$ There exists a constant $A_1>0$ such that the inequality  
	\begin{equation*}
	\int_{\mathbb{R}^n} ({I}_\alpha  h) ^r  d\mu \leq A_1 \|h\|_{L^s(\mathbb{R}^n)}^{r}
	\end{equation*} 
	holds for all nonnegative $h\in L^s(\mathbb{R}^n)$.

	$\rm{(ii)}$ $\mu$ is continuous w.r.t ${\rm cap}_{\alpha,s}$ and for any quasi-continuous or $\mu$-measurable function $u\in L^{\frac{s}{r}}({\rm cap}_{\alpha,s})$, we have 
	\begin{equation}\label{dual-f}
	\left|\int_{\mathbb{R}^n} u d\mu\right| \leq A_2 \|u\|_{L^{\frac{s}{r}}({\rm cap}_{\alpha,s})} 
	\end{equation}
for a constant $A_2>0$.
	
	$\rm{(iii)}$ $$[W_{\alpha,s} \mu]^{s-1}\in L^\frac{r}{s-r}(d\mu).$$
	
	$\rm{(iv)}$  $$ [W_{\alpha ,s}\mu]^{s-1} \in L^{\frac{s}{s-r}}({\rm cap}_{\alpha,s}).$$
Moreover, we have 
$$A_1\simeq A_2 \simeq \left \|[W_{\alpha,s} \mu]^{s-1}\right \|_{L^\frac{r}{s-r}(d\mu)}^{\frac{r}{s}} \simeq \left \|[W_{\alpha ,s}\mu]^{s-1} \right\|_{L^{\frac{s}{s-r}}({\rm cap}_{\alpha,s})}.$$

\end{theorem}

\begin{proof} ${\rm (i)} \Rightarrow {\rm (ii)}:$ Let $N$ be a set such that ${\rm cap}_{\alpha, s}(N)=0$. Choose a $G_\delta$ set $\widetilde{N}$ such that 
	${\rm cap}_{\alpha, s}(\widetilde{N})=0$ and $N\subseteq \widetilde{N}$. This is possible since 
	$$0={\rm cap}_{\alpha, s}(N)=\inf\{{\rm cap}_{\alpha, s}(G):  N\subseteq G, G {\rm ~ open} \}.$$
	If $h\in L^s(\mathbb{R}^n), h\geq 0$, is such that $I_\alpha h\geq 1$ on 	$\widetilde{N}$, then it follows from ${\rm (i)}$ that 
	$$\mu(\widetilde{N})\leq A_1 \|h\|_{L^s}^r.$$
	Taking the infimum over such functions $h$ we arrive at 
	$$\mu(\widetilde{N})\leq A_1\, {\rm cap}_{\alpha, s}(\widetilde{N})^{r/s}=0.$$
	By identifying $\mu$ with its completion we thus see that $N$ is $\mu$-measurable and $\mu(N)=0$. This proves that 
	$\mu$ is continuous w.r.t ${\rm cap}_{\alpha,s}$.
	
	Now let $u$ be a  quasi-continuous function in $L^{\frac{s}{r}}({\rm cap}_{\alpha,s})$. Then $u$ is $\mu$-measurable.
	For any $h\in L^s$, $h\geq0$, such that $(I_\alpha h)^r\geq |u|$ q.e., by ${\rm (i)}$ we have 
	\begin{equation*}
	\left|\int_{\mathbb{R}^n} u d\mu\right| \leq A_1 \|h\|_{L^s}^r 
	\end{equation*}
	Taking the infimum over such $h$ and applying Theorem \ref{Fnorm}  we obtain \eqref{dual-f}.
	
	\vspace{.1in}
	
	\noindent ${\rm (ii)} \Rightarrow {\rm (iii)}:$ Let $h\in L^s$ and applying \eqref{dual-f} with $u=(I_\alpha f)^r$, we find 
	\begin{align*}
	\left(\int_{\mathbb{R}^n} ({I}_\alpha  h) ^r  d\mu\right)^{\frac{1}{r}} \leq A_2^{\frac{1}{r}} \|I_\alpha h\|_{L^s({\rm cap}_{\alpha,s})}\leq c\, A_2^{\frac{1}{r}} \|h\|_{L^s},
	\end{align*}
	where the last inequality follows from \eqref{CSIM}. Thus, by \cite[Theorem 1.13]{Ver1} (see also \cite{COV}) we obtain  that $[W_{\alpha,s} \mu]^{\frac{r(s-1)}{s-r}}\in L^1(d\mu).$
	
	\vspace{.1in}
	
	\noindent ${\rm (iii)} \Rightarrow {\rm (iv)}:$ For any $t>0$, let 
	$E_t=\{W_{\alpha,s}\mu>t\}$. We first claim that there are constants $a, A\geq 1$ such that 
	\begin{align}\label{onEt}
	{\rm cap}_{\alpha,s}(E_{a \cdot t})\leq A  t^{1-s} \mu(E_t).
	\end{align}
	Indeed, set $\mu^t=\mu\chi_{E_t}$ and $\mu_t=\mu\chi_{\mathbb{R}^n\setminus E_t}$. Since $\mathbb{R}^n\setminus E_t$ is a closed set in $\mathbb{R}^n$, we have 
	${\rm supp}(\mu_t)\subset \mathbb{R}^n\setminus E_t$. Thus,  $W_{\alpha,s}\mu_t \leq W_{\alpha,s}\mu\leq t$ on ${\rm supp}(\mu_t)$. It then follows from the boundedness principle, Lemma \ref{bounded}, that $W_{\alpha,s}\mu_t\leq c \, t $ on the whole $\mathbb{R}^n$, with a constant $c\geq 1$. Thus, there exists a constant $c_1\geq1$ such that 
	\begin{align*}
	\{W_{\alpha,s}\mu >2 c\, c_1 t\} &\subseteq \{c_1 [W_{\alpha,s}(\mu^t) + W_{\alpha,s}(\mu_t)] > 2 c\, c_1 t \} \\
	&\subseteq \{  W_{\alpha,s}(\mu^t) > c\,  t \} \bigcup \{  W_{\alpha,s}(\mu_t) > c\,  t \}\\
&\subseteq \{  W_{\alpha,s}(\mu^t) > c\,  t \}.	
	\end{align*}
	By \cite[Proposition 6.3.12]{AH}, this gives 
	\begin{align*}
	{\rm cap}_{\alpha,s}(\{W_{\alpha,s}\mu > 2  c\, c_1 t\}) \leq A  t^{1-s} \mu^t(\mathbb{R}^n)=A  t^{1-s} \mu(E_t),
	\end{align*}
	which is \eqref{onEt} with $a=2 c \, c_1$. Now using \eqref{onEt}, we get 
	\begin{align*}
	\int_{\mathbb{R}^n} [W_{\alpha,s}\mu]^{\frac{(s-1)s}{s-r}} d{\rm cap}_{\alpha,s} &=\frac{(s-1)s}{s-r}\int_0^\infty  (a\cdot t)^{\frac{(s-1)s}{s-r}}{\rm cap}_{\alpha,s}(E_{a\cdot t}) \frac{dt}{t}\\
	&\leq C \int_0^\infty  t^{\frac{(s-1)s}{s-r}+1-s}\mu(E_t) \frac{dt}{t}\\
	&= C \int_{\mathbb{R}^n} [W_{\alpha,s} \mu]^{\frac{r(s-1)}{s-r}} d\mu. 
	\end{align*}
	
	\vspace{.1in}
	
	\noindent ${\rm (iv)} \Rightarrow {\rm (i)}:$ Let $\nu$ be a measure defined by 
	$$d\nu=\frac{d\mu}{W_{\alpha,s}(\mu)^{s-1}}.$$
	Then we have $$\nu(E)\leq c\, {\rm cap}_{\alpha,s}(E)$$
	for any Borel set $E\subset\mathbb{R}^n$ (see \cite[Theorem 1.11]{Ver1}).	 Thus, it follows that 
	\begin{align*}
	\int_{\mathbb{R}^n} [W_{\alpha,s} \mu]^{\frac{r(s-1)}{s-r}} d\mu&= \int_{\mathbb{R}^n} [W_{\alpha,s} \mu]^{\frac{s(s-1)}{s-r}} d\nu \\
	&= \frac{s(s-1)}{s-r}\int_0^\infty t^{\frac{s(s-1)}{s-r}} \nu (\{W_{\alpha,s} \mu>t\}) \frac{dt}{t} \\
	& \leq c  \frac{s(s-1)}{s-r}\int_0^\infty t^{\frac{s(s-1)}{s-r}} {\rm cap}_{\alpha,s} (\{W_{\alpha,s} \mu>t\}) \frac{dt}{t}\\
	& = c \int_{\mathbb{R}^n} [W_{\alpha,s}\mu]^{\frac{(s-1)s}{s-r}} d{\rm cap}_{\alpha,s}.
	\end{align*}
By \cite[Theorem 1.13]{Ver1}, this gives $\rm (i)$, which completes the proof.
\end{proof}

Note that Theorem \ref{upper-tri} allows to rewrite the norm of a function $f\in \dot{M}^{\alpha,s}_{p,r}$, $0<r<s$, as 
\begin{equation*} 
\|f\|_{\dot{M}^{\alpha,s}_{p,r}}\simeq \| W_{\alpha ,s}(|f|^p)^{\frac{s-1}{p}}\|_{L^{\frac{sp}{s-r}}({\rm cap}_{\alpha,s})}. 
\end{equation*}
Moreover, $f\in \dot{M}^{\alpha,s}_{p,r}$ if and only if  $|f|^p$  belongs to  the dual of the space of quasi-continuous functions in $L^{\frac{s}{r}}({\rm cap}_{\alpha,s})$ (see \cite[Theorem 9]{Ada}).

On the other hand, for $r=s$, the norm of a function $f\in \dot{M}^{\alpha,s}_{p, s}$ is equivalent to the quantity
\begin{align*}
\sup_{K}\left(\frac{\int_{K}|f(x)|^{p}dx}{\text{cap}_{\alpha,s}(K)}\right)^{1/p},
\end{align*}
where the supremum is taken over all compact sets $K\subset{\mathbb{R}}^{n}$ with non-zero capacity. Equivalently, in this case $f\in \dot{M}^{\alpha,s}_{p, s}$ if and only if $|f|^p$ belongs to  the dual of the space of quasi-continuous functions in $L^{1}({\rm cap}_{\alpha,s})$ (see \cite{Ada, OP1}).

\begin{theorem}\label{secondtheorem}
	For  $s>1$, $0<\alpha< n/s$, $p>1$, and $0<r\leq s$, it holds that 
	\begin{align*}
	\|f\|_{\dot{M}^{\alpha,s}_{p,r}}\simeq \sup_{w}\left(\int_{\mathbb{R}^{n}}|f(x)|^{p} w(x)dx\right)^{1/p},
	\end{align*}
	where the supremum is taken over all nonnegative  measurable and quasi-continuous $w\in L^{s/r}({\rm cap}_{\alpha,s})$ such that $\|w\|_{L^{s/r}({\rm cap}_{\alpha,s})}\leq 1$ and 
	$[w]_{A_1}\leq \bar{\bf c}(n,\alpha):=[\left|\cdot\right|^{\alpha-n}]_{A_1}$.
\end{theorem}

\begin{proof}
 If $0<r<s$, then by Theorem \ref{upper-tri} and the definition of $\dot{M}_{p,r}^{\alpha,s}$, for  any q.e. defined weight $w$ with $\|w\|_{L^{s/r}({\rm cap}_{\alpha,s})}\leq 1$, it holds that 
	\begin{align*}
\left(\int_{\mathbb{R}^{n}}|f(x)|^{p} w(x)dx\right)^{1/p}\leq c \cdot \|f\|_{\dot{M}_{p,r}^{\alpha,s}}.
\end{align*}
On the other hand, if $r=s$, then from the characterization
$$\|f\|_{\dot{M}^{\alpha,s}_{p,s}}\simeq \sup_{K}\left(\frac{\int_{K}|f(x)|^{p}dx}{\text{cap}_{\alpha,s}(K)}\right)^{1/p},$$
we also have, for any weight $w$ with $\|w\|_{L^{1}({\rm cap}_{\alpha,s})}\leq 1$,
	\begin{align*}
\left(\int_{\mathbb{R}^{n}}|f(x)|^{p} w(x)dx\right)^{1/p}&= \left(\int_{0}^\infty \int_{\{w>t\}}|f|^{p} dx dt \right)^{1/p} \\
&\leq c \cdot \|f\|_{\dot{M}_{p,r}^{\alpha,s}} \left(\int_{0}^\infty  {\rm cap}_{\alpha,s}(\{w>t\}) dt \right)^{1/p}\\
&\leq c \cdot \|f\|_{\dot{M}_{p,r}^{\alpha,s}}.
\end{align*}

	Taking supremum over all such weights $w$, we obtain, for $0<r\leq s$, 
	\begin{align*}
	\sup_{w}\left(\int_{\mathbb{R}^{n}}|f(x)|^{p} w(x)dx\right)^{1/p}\leq c\cdot\|f\|_{\dot{M}_{p,r}^{\alpha,s}}.
	\end{align*}
	Note that the  quasi-continuity and $A_{1}$ conditions are not used in the above estimate. Now we show for the other direction of the estimates. We first consider the case where $0<r\leq 1$. Denote by
	\begin{align*}
	\mathfrak{M}=\sup\left\{\left(\int_{\mathbb{R}^{n}}(I_{\alpha}h)^{r}|f|^{p}dx\right)^{\frac{1}{p}}:h\geq 0,~ \|h\|_{L^{s}(\mathbb{R}^{n})}\leq 1\right\}.
	\end{align*}
	Then by the definition of $\dot{M}_{p,r}^{\alpha,s}$, one has 
	\begin{align*}
	\|f\|_{\dot{M}_{p,r}^{\alpha,s}}=\mathfrak{M}.
	\end{align*}
	Also note that $(I_{\alpha}h)^{r}\in A_{1}$ for $0<r\leq 1$ and $h\in L^s(\mathbb{R}^n)$ with $A_{1}$ constant satisfying 
	\begin{align*}
	[(I_{\alpha}h)^{r}]_{A_{1}}\leq[I_{\alpha}h]_{A_{1}}\leq [\left|\cdot\right|^{\alpha-n}]_{A_1}={\bar{\bf c}}(n,\alpha).
	\end{align*}
	Moreover, $(I_{\alpha}h)^{r}$ is quasi-continuous and by \eqref{CSIM},
	\begin{align*} 
	\|(I_{\alpha}h)^{r}\|_{L^{s/r}({\rm cap}_{\alpha,s})}&=\left(\int_{\mathbb{R}^{n}}(I_{\alpha}h)^{s}d{\rm cap}_{\alpha,s}\right)^{\frac{r}{s}}\\
	&\leq c\cdot\|h\|_{L^{s}(\mathbb{R}^{n})}^{\frac{r}{s}}\leq c \nonumber
	\end{align*}
	for any function $h\geq 0$ with $\|h\|_{L^{s}(\mathbb{R}^{n})}\leq 1$. We thus conclude that 
	\begin{align*}
	\mathfrak{M}\leq c\cdot\sup_{w}\left(\int_{\mathbb{R}^{n}}|f(x)|^{p} w(x)dx\right)^{1/p},
	\end{align*}
	where the supremum is taken over all nonnegative  measurable and quasi-continuous $w\in L^{s/r}({\rm cap}_{\alpha,s})$ such that $\|w\|_{L^{s/r}({\rm cap}_{\alpha,s})}\leq 1$
	 and 
	$[w]_{A_1}\leq \bar{\bf c}(n,\alpha)$.
	
	Now we consider the case where $r>1$. Fix $h\geq 0$ with $\|h\|_{L^{s}(\mathbb{R}^{n})}\leq 1$.  
	Then as in \eqref{hIhO} we have that $h\cdot(I_{\alpha}h)^{r-1}\in \dot{\widetilde{O}}_{s/r}^{\alpha,s}$.
	Let $w=I_{\alpha}(h\cdot(I_{\alpha}h)^{r-1})$. Then $w\in A_{1}$ with $[w]_{A_1}\leq [\left|\cdot\right|^{\alpha-n}]_{A_1}={\bar{\bf c}}(n,\alpha)$, and by Proposition \ref{conv-quasi} below
we have that 	$w$ is quasi-continuous. 

Now	with $h_{N}=h\cdot\chi_{\{|h|\leq N\}}\cdot \chi_{\{|x|\leq N\}}$, $N\in\mathbb{N}$, we find 
		\begin{align*}
	\|w\|_{L^{s/r}({\rm cap}_{\alpha,s})}^{s/r}&=\left\|I_{\alpha}\left(h\cdot(I_{\alpha}h)^{r-1}\right)\right\|_{L^{s/r}({\rm cap}_{\alpha,s})}^{s/r}\\
	&=\sup_{N}\int_{\mathbb{R}^{n}}\left[I_{\alpha}\left(h_{N}\cdot(I_{\alpha}h_{N})^{r-1}\right)\right]^{s/r}d{\rm cap}_{\alpha,s}\\
	&\leq c\cdot\sup_{N}\int_{\mathbb{R}^{n}}h_{N}^{s}(I_{\alpha}h_{N})^{(r-1)s}\left[I_{\alpha}\left(h_{N}\cdot(I_{\alpha}h_{N})^{r-1}\right)\right]^{s/r-s}dx\\
	&\leq c\cdot\sup_{N}\int_{\mathbb{R}^{n}}h_{N}^{s}(I_{\alpha}h_{N})^{(r-1)s}(I_{\alpha}h_{N})^{r(s/r-s)}dx\\
	&=c\cdot\sup_{N}\int_{\mathbb{R}^{n}}h_{N}^{s}dx\leq c.
	\end{align*}
Note that we applied Theorem \ref{Main1} in the first inequality. In the second inequality, we used that	$\frac{s}{r}-s<0$ and Lemma \ref{IBPL}. Also, by Lemma \ref{IBPL}, we have  $(I_{\alpha}h)^{r}\leq c\cdot w$, and thus
	\begin{align*}
	\left(\int_{\mathbb{R}^{n}}(I_{\alpha}h)^{r}|f|^{p}dx\right)^{\frac{1}{p}} \leq c\left(\int_{\mathbb{R}^{n}}|f(x)|^{p} w(x)dx\right)^{1/p},
	\end{align*}
	which  completes the proof of the theorem.
\end{proof}

\begin{proposition}\label{conv-quasi}
	Suppose that $\{\varphi_{i}\}_{i=1}^{\infty}$ converges to $\varphi$ in $\dot{\widetilde{O}}_{q}^{\alpha,s}$, $1\leq q<s$. Then there is a subsequence $\{\varphi_{i_{k}}\}_{k=1}^{\infty}$ such that $\lim\limits_{k\rightarrow\infty}I_{\alpha}\varphi_{i_{k}}(x)=I_{\alpha}\varphi(x)$ q.e., and uniformly outside an open set of arbitrarily small ${\rm cap}_{\alpha,s}(\cdot)$. As a consequence, $I_{\alpha}\varphi$ is quasi-continuous for any $\varphi$ in $\dot{\widetilde{O}}_{q}^{\alpha,s}$.
\end{proposition}
\begin{proof}
		Note that as consequence of Theorem \ref{Main2}, one has 
	\begin{align*}
	{\rm cap}_{\alpha,s}(\{I_{\alpha}\varphi>\lambda\})\leq c\cdot\frac{1}{\lambda}\cdot\|\varphi\|_{\dot{\widetilde{O}}_{s,q}^{\alpha,s}},\quad\lambda>0.
	\end{align*}
	Using the above estimate, one can easily modify the argument of \cite[Proposition 2.3.8]{AH} to obtain the  desired subsequence. To prove the last statement 
	let $\varphi\in \dot{\widetilde{O}}_{q}^{\alpha,s}$ and set $\varphi_{N}=\varphi\cdot\chi_{\{|\varphi|\leq N\}}\cdot \chi_{\{|x|\leq N\}}, N\in \mathbb{N}$.
	Then an argument by Lebesgue dominated convergence theorem will show that $\varphi_{N}\rightarrow\varphi$ in $\dot{\widetilde{O}}_{q}^{\alpha,s}$.
	On the other hand, each $I_{\alpha}\varphi_{N}$ is continuous and thus applying the first part of the proposition we obtain the quasi-continuity of $I_{\alpha}\varphi$.
\end{proof}

The following corollary is immediate by re-examining the proof of Theorem \ref{secondtheorem}.
\begin{corollary}
	For  $s>1$, $0<\alpha< n/s$, $p>1$, and $0<r\leq s$, it holds that 
	\begin{align*}
	\|f\|_{\dot{M}^{\alpha,s}_{p,r}}\simeq \sup_{w}\left(\int_{\mathbb{R}^{n}}|f(x)|^{p} w(x)dx\right)^{1/p},
	\end{align*}
	where the supremum is taken over all nonnegative  measurable and q.e. defined  $w\in L^{s/r}(C)$ with $\|w\|_{L^{s/r}(C)}\leq 1$.
\end{corollary}

\section{The spaces $\dot{N}^{\alpha,s}_{p', s/r}$ and $(\dot{N}^{\alpha,s}_{p', s/r})'$ and Proof of Theorem \ref{Main3}} 

Let $s>1$,   $0<\alpha<\frac{n}{s}$,  $0< r\leq s$, and $p>1$. Recall that  $\dot{N}^{\alpha,s}_{p', s/r}$ is defined as the space of all measurable functions $g$ in
$\mathbb{R}^n$ such that  $\|g\|_{\dot{N}^{\alpha,s}_{p', s/r}}<+\infty$. Here
$$\|g\|_{\dot{N}^{\alpha,s}_{p', s/r}} = \inf_{w}\left(\int_{\mathbb{R}^n} |g|^{p'} w^{1-p'} dx   \right)^{\frac{1}{p'}},$$
where the infimum is taken over all quasi-continuous weights $w\in A_{1}$ with $\|w\|_{L^{s/r}({\rm cap}_{\alpha,s})}\leq 1$, $[w]_{A_1}\leq \bar{\bf c}(n,\alpha)$, and we can take 
$\bar{\bf c}(n,\alpha)=[\left|\cdot\right|^{\alpha-n}]_{A_1}$. On the other hand, if we do not use the quasi-continuity and $A_1$ conditions then we have, by definition, the space    $\dot{\widetilde{N}}^{\alpha,s}_{p', s/r}$.

In a similar manner, we now define  $\dot{\overline{N}}^{\alpha,s}_{p', s/r}=\dot{\overline{N}}^{\alpha,s}_{p', s/r}(\mathbb{R}^n)$  as the space of all measurable functions $g$ in
$\mathbb{R}^n$ such that  $\|g\|_{\dot{\overline{N}}^{\alpha,s}_{p', s/r}}<+\infty$, where 
\begin{align*}
\|g\|_{\dot{\overline{N}}_{p',s/r}^{\alpha,s}}=\inf_{w}\left(\int_{\mathbb{R}^n} |g|^{p'} w^{1-p'} dx   \right)^{\frac{1}{p'}},
\end{align*}
and the infimum is taken over all q.e. defined weights $w$ such that $\|w\|_{\dot{F}}\leq 1$. Here $\dot{F}=\dot{F}^{\alpha,s}_r$ is as in Theorem \ref{Fnorm}.

\begin{theorem}\label{3spa}
	For $s>1$, $0<\alpha<n/s$, and $0< r\leq s$, it holds that 
	\begin{equation*} 
	\dot{N}^{\alpha,s}_{p', s/r} \approx \dot{\widetilde{N}}^{\alpha,s}_{p', s/r} \approx \dot{\overline{N}}^{\alpha,s}_{p', s/r}. 
	\end{equation*}
\end{theorem}

\begin{proof}
Observe that by Theorem \ref{Fnorm} we have that 
$\dot{\widetilde{N}}^{\alpha,s}_{p', s/r} \approx \dot{\overline{N}}^{\alpha,s}_{p', s/r}$.
	Moreover, it is obvious that  $\dot{N}_{p',s/r}^{\alpha,s} \hookrightarrow \dot{\widetilde{N}}_{p',s/r}^{\alpha,s}$.
Thus we just need  to show 
\begin{equation}\label{LTS}
\dot{\widetilde{N}}_{p',s/r}^{\alpha,s} \hookrightarrow   \dot{N}_{p',s/r}^{\alpha,s}.
\end{equation}
Indeed, for any $f\in \dot{\widetilde{N}}_{p',s/r}^{\alpha,s}$ and $\varepsilon>0$, there exists a weight $w_0$ such that $\|w_0\|_{\dot{F}}\leq 1$ and 
\begin{equation}\label{epin}
\left(\int_{\mathbb{R}^n} |f|^{p'} (w_0)^{1-p'} dx\right)^{\frac{1}{p'}} \leq  \|f\|_{\dot{\widetilde{N}}_{p',s/r}^{\alpha,s}} +\varepsilon.
\end{equation}
We now use Lemma \ref{extreme} to  choose an $0\leq h\in L^{s}(\mathbb{R}^{n})$ such that  $I_{\alpha}h\geq (w_0)^{\frac{1}{r}}$ q.e. and
$\|h\|_{L^{s}(\mathbb{R}^{n})}^{r}\leq 1$. Let $w=(I_{\alpha}h)^r$ if $0<r\leq 1$ and  $w=I_{\alpha}\left(h\cdot(I_{\alpha}h)^{r-1}\right)$ if $1<r\leq s$. Then as in the proof of Theorem \ref{secondtheorem}, we have $[w]_{A_1}\leq [\left|\cdot\right|^{\alpha-n}]_{A_1}$, $w$ is quasi-continuous,  	$\|w\|_{L^{s/r}({\rm cap}_{\alpha,s})}\leq c$, and   
$$w_0\leq (I_{\alpha}h)^r \leq c\cdot w \quad {\rm q.e.}$$
Hence,
	\begin{align}\label{w0l}
	\int_{\mathbb{R}^{n}}|f|^{p'} (w_0)^{1-p'}dx\geq c \int_{\mathbb{R}^{n}}|f|^{p'}w^{1-p'}dx.
	\end{align}
Then it follows from \eqref{epin} and \eqref{w0l} that 
	\begin{align*}
	\|f\|_{\dot{\widetilde{N}}_{p',s/r}^{\alpha,s}} +\varepsilon \geq c\cdot\|f\|_{\dot{N}_{p',s/r}^{\alpha,s}}.
	\end{align*}
Finally, letting $\varepsilon\rightarrow0$, we obtain \eqref{LTS} as desired.
\end{proof}

The reason to consider $\dot{\overline{N}}^{\alpha,s}_{p', s/r}$ is that it is always a Banach function space, whereas $\dot{\widetilde{N}}^{\alpha,s}_{p', s/r}$ and 
$\dot{N}^{\alpha,s}_{p', s/r}$ may not. 

\begin{theorem}\label{switch}
	For $s>1$, $0<\alpha<n/s$, $0<r\leq s$, and $p>1$, the space $\dot{\overline{N}}_{p',s/r}^{\alpha,s}$ is a Banach function space.
\end{theorem}
\begin{proof}
	Let 
	\begin{align*}
	\|g\|'=\inf\left\{\sum_{j}|c_{j}|: g=\sum_{j}c_{j}b_{j}~{\rm a.e.},~\|b_{j}\|_{\dot{\overline{N}}_{p',s/r}^{\alpha,s}}\leq 1\right\},
	\end{align*}
	where $g\in \dot{\overline{N}}_{p',s/r}^{\alpha,s}$. It is easy to see that $\left\|\cdot\right\|'$ is a seminorm and $(\dot{\overline{N}}_{p',s/r}^{\alpha,s},\left\|\cdot\right\|')$ is a complete semi-normed space. We now show that $\|g\|'=0$ entails $g=0$ a.e. so that $(\dot{\overline{N}}_{p',s/r}^{\alpha,s},\left\|\cdot\right\|')$ is a Banach space. Let $\varepsilon>0$ be given. If $\|g\|'=0$, then there exist $\{c_{j}\}\in\ell^{1}$ and $b_{j}$ with $\|b_{j}\|_{\dot{\overline{N}}_{p',s/r}^{\alpha,s}}\leq 1$ such that $g=\sum_{j}c_{j}b_{j}$ and $\sum_{j}|c_{j}|<\varepsilon$. Let $w_{j}$ be a weight such that $\|w_{j}\|_{\dot{F}}\leq 1$ and  
		\begin{align*}
	\left(\int_{\mathbb{R}^{n}}|b_{j}|^{p'}w_{j}^{1-p'}dx\right)^{\frac{1}{p'}}\leq 2.
	\end{align*} 
	Then for any $\varphi\in C_{c}(\mathbb{R}^{n})$ by H\"older's inequality it holds that   
	\begin{align*}
	\int_{\mathbb{R}^{n}}|g||\varphi|dx&\leq\sum_{j}|c_{j}|\int_{\mathbb{R}^{n}}|b_{j}||\varphi|dx\\
	&\leq 2 \sum_{j}|c_{j}| \left(\int_{\mathbb{R}^n} |\varphi|^p w_j dx\right)^{\frac1p} \\
	&\leq2 \varepsilon \|\varphi\|_{L^\infty(\mathbb{R}^n)} \sup_{j}\left(\int_{{\rm supp (\varphi)}}  w_j dx\right)^{\frac1p}.
	\end{align*}
	Recall that $|E|^{1-\frac{\alpha s}{n}}\leq c\cdot{\rm cap}_{\alpha,s}(E)$ for any measurable set $E\subseteq\mathbb{R}^n$, and thus we find (see, e.g., \cite[Equ. (3.1)]{OP3})
\begin{equation}\label{cap-Sobolev}	
\| w_j\|_{L^{\frac{ns}{r(n-\alpha s)}}(\mathbb{R}^n, dx)} \leq c_1 \| w_j\|_{L^{s/r}({\rm cap}_{\alpha, s } )} \leq c_2 \| w_j\|_{\dot{F}} \leq c_2.
\end{equation}
Then by H\"older's inequality and letting $\varepsilon\rightarrow 0$, we see that    
	\begin{align*}
	\int_{\mathbb{R}^{n}}|g||\varphi|dx=0
	\end{align*}
	for all $\varphi\in C_{c}(\mathbb{R}^{n})$ and hence $g=0$ a.e. as expected. 
	
	Now we show that 
	\begin{align}\label{5.2}
	\|g\|_{\dot{\overline{N}}_{p',s/r}^{\alpha,s}}=\|g\|',\quad g\in \dot{\overline{N}}_{p',s/r}^{\alpha,s}.
	\end{align}
	Let $\varepsilon>0$ be given. By expressing $g$ as 
	\begin{align*}
	g=\left(\|g\|_{\dot{\overline{N}}_{p',s/r}^{\alpha,s}}+\varepsilon\right)\cdot\frac{g}{\|g\|_{\dot{\overline{N}}_{p',s/r}^{\alpha,s}}+\varepsilon},
	\end{align*}
	it follows that 
	\begin{align*}
	\|g\|'\leq\|g\|_{\dot{\overline{N}}_{p',s/r}^{\alpha,s}}+\varepsilon.
	\end{align*}
	The arbitrariness of $\varepsilon>0$ yields that  $\|g\|'\leq\|g\|_{\dot{\overline{N}}_{p',s/r}^{\alpha,s}}$.
	
	On the other hand,  let $g\in \dot{\overline{N}}_{p',s/r}^{\alpha,s} $ be such that 
	\begin{align}\label{combi}
	g=\sum_{j}c_{j}b_{j},
	\end{align}
	where $\{c_{j}\}\in\ell^{1}$ and $\|b_{j}\|_{\dot{\overline{N}}_{p',s/r}^{\alpha,s}}\leq 1$. Choose a weight $w_{j}$ such that $\|w_{j}\|_{\dot{F}}\leq 1$ and 
	\begin{align*}
	\|b_{j}\|_{\dot{\overline{N}}_{p',s/r}^{\alpha,s}}\leq\left(\int_{\mathbb{R}^{n}}|b_{j}|^{p'}w_{j}^{1-p'}dx\right)^{\frac{1}{p'}}<1+\varepsilon.
	\end{align*} 
	Let 
	\begin{align*}
	w=\|\{c_{j}\}\|_{\ell^{1}}^{-1}\sum_{j}|c_{j}|\cdot w_{j}.
	\end{align*}
	Then Theorem \ref{Fnorm} implies
	\begin{align*}
	\|w\|_{\dot{F}}\leq\|\{c_{j}\}\|_{\ell^{1}}^{-1}\sum_{j}|c_{j}|\cdot\| w_{j}\|_{\dot{F}}\leq \|\{c_{j}\}\|_{\ell^{1}}^{-1}\sum_{j}|c_{j}|\leq 1.
	\end{align*}
	On the other hand, H\"{o}lder's inequality gives 
	\begin{align*}
	|g|^{p'}\leq\left(\sum_{j}|c_{j}||b_{j}|\right)^{p'} \leq \left(\sum_{j}|c_{j}|\cdot w_{j} \right)^{p'-1}\left(\sum_{j}|c_{j}|\cdot |b_{j}|^{p'} w_{j}^{1-p'}\right).
	\end{align*}
	Hence,
		\begin{align}\label{gpw1p}
	|g|^{p'} w^{1-p'}\leq\|\{c_{j}\}\|_{\ell^{1}}^{p'-1}\left(\sum_{j}|c_{j}|\cdot  |b_{j}|^{p'} w_{j}^{1-p'} \right),
	\end{align}
	so that 
	\begin{align*}
	\int_{\mathbb{R}^{n}}|g|^{p'} w^{1-p'}dx&\leq\|\{c_{j}\}\|_{\ell^{1}}^{p'-1}\left(\sum_{j}|c_{j}|\int_{\mathbb{R}^{n}}|b_{j}|^{p'} w^{1-p'}dx\right)\\
	&\leq\|\{c_{j}\}\|_{\ell^{1}}^{p'-1}\sum_{j}|c_{j}|\cdot(1+\varepsilon)^{p'}\\
	&=\|\{c_{j}\}\|_{\ell^{1}}^{p'}\cdot(1+\varepsilon)^{p'}.
	\end{align*}
	Thus,
	\begin{align*}
	\|g\|_{\dot{\overline{N}}_{p',s/r}^{\alpha, s}}\leq \left(\int_{\mathbb{R}^{n}}|g|^{p'}w^{1-p'}dx\right)^{\frac{1}{p'}}\leq (1+\varepsilon)\sum_{j}|c_{j}|.
	\end{align*} 
	Taking the infimum with respect to all   representations of $g$ in the form \eqref{combi}, we have 
	\begin{align*}
	\|g\|_{\dot{\overline{N}}_{p',s/r}^{\alpha, s}}\leq(1+\varepsilon)\cdot\|g\|'.
	\end{align*}
	The arbitrariness of $\varepsilon>0$ yields (\ref{5.2}), which also shows that $\dot{\overline{N}}_{p',s/r}^{\alpha,s}$ is a Banach space.
	
	To  show that $\dot{\overline{N}}_{p',s/r}^{\alpha, s}$ is a  Banach function space we just need to show  property (P2), as properties (P3) and (P4) are easy to show. To this end, let $\{g_{j}\}_{j=1}^{\infty}\subseteq \dot{\overline{N}}_{p',s/r}^{\alpha, s}$ be a sequence of nonnegative measurable functions such that $g_{j}\uparrow g$ a.e. in $\mathbb{R}^{n}$. We are to show that 
	\begin{align}\label{fatou}
	\|g\|_{\dot{\overline{N}}_{p',s/r}^{\alpha, s}}\leq\sup_{j}\|g_{j}\|_{\dot{\overline{N}}_{p',s/r}^{\alpha, s}}.
	\end{align}
	We may assume without loss of generality that 
	\begin{align*}
	\sup_{j}\|g_{j}\|_{\dot{\overline{N}}_{p',s/r}^{\alpha, s}}=M<+\infty.
	\end{align*}
	Then for any $j\in\mathbb{N}$ and $\varepsilon>0$, there exists a weight $w_{j}$ such that $\|w_{j}\|_{\dot{F}}\leq 1$ and
	\begin{align}\label{sta}
	\left(\int_{\mathbb{R}^{n}}g_{j}^{p'} w_{j}^{1-p'}dx\right)^{\frac{1}{p'}}\leq M+\varepsilon.
	\end{align}
Note that  by \eqref{cap-Sobolev} and  H\"older's inequality it follows that, for, say, $v(x)=e^{-|x|}$,   
$$\| w_j\|_{L^{1}(\mathbb{R}^n, v)} \leq C.$$
Thus  Koml\'os Theorem (see \cite{Kom}) yields a  subsequence of $\{ w_{j}\}$,  still denoted by $\{w_{j}\}$, and a nonnegative measurable function $w$ such that  $$\sigma_{k}(x):=\dfrac{1}{k}\sum_{j=1}^{k}w_{j}(x)\rightarrow w(x) \quad {\rm a.e. ~} x\in\mathbb{R}^n.$$  Moreover, any subsequence of $\{w_j\}$ is also Ces\`aro convergent to $w$ almost everywhere. By redefining $w(x)$ to be zero for all the points $x$ such that  $\sigma_{k}(x)\not\rightarrow w(x)$, one has 
\begin{align*}
w(x)\leq\liminf_{k\rightarrow\infty}\sigma_{k}(x) \qquad \text{q.e.}
\end{align*} 
Then by Proposition \ref{the F norm}, we have
\begin{align}\label{w-norm}
\| w\|_{\dot{F}}\leq\left\|\liminf_{k\rightarrow\infty}\sigma_{k}\right\|_{\dot{F}}\leq\liminf_{k\rightarrow\infty}\|\sigma_{k}\|_{\dot{F}} \leq \liminf_{k\rightarrow\infty} \frac1k \sum_{j=1}^k\|w_{j}\|_{\dot{F}} \leq 1.
\end{align}
On the other hand,  by the convexity of the function 
$t \mapsto t^{1-p'}$ on $(0,\infty)$, it holds that, for any $j\in \mathbb{N}$,

\begin{align*}
\int_{\mathbb{R}^n} g_{j}(x)^{p'} w(x)^{1-p'}dx& =  \int_{\mathbb{R}^n} g_{j}(x)^{p'}  \lim_{k\rightarrow\infty} \Big[\dfrac{1}{k}\sum_{m=j}^{j+k-1} w_{m}(x)\Big]^{1-p'}dx\\
&\leq \liminf_{k\rightarrow\infty}\int_{\mathbb{R}^n} g_{j}(x)^{p'} \Big[\dfrac{1}{k}\sum_{m=j}^{j+k-1} w_{m}(x)\Big]^{1-p'}dx\\
&\leq\liminf_{k\rightarrow\infty}\int_{\mathbb{R}^n} g_{j}(x)^{p'}\,  \dfrac{1}{k}\sum_{m=j}^{j+k-1} w_{m}(x)^{1-p'}dx\\
&\leq\liminf_{k\rightarrow\infty}\int_{\mathbb{R}^n}   \dfrac{1}{k}\sum_{m=j}^{j+k-1} g_{m}(x)^{p'} \, w_{m}(x)^{1-p'}dx,
\end{align*}
where we used $0\leq g_j\leq g_m$ for $m\geq j$ in the last inequality. At this point, we use \eqref{sta} to obtain
\begin{align*}
\int g_{j}(x)^{p'} w(x)^{1-p'}dx \leq (M+\epsilon)^{p'}.
\end{align*}
Then sending $j\rightarrow\infty$ and $\epsilon\rightarrow 0$, by Fatou's lemma we get
\begin{align*}
\int g(x)^{p'} w(x)^{1-p'}dx \leq M^{p'}.
\end{align*}
This and \eqref{w-norm} now give \eqref{fatou} and complete the proof of the theorem.
\end{proof}

\begin{remark} Let $\dot{\overline{\overline{N}}}_{p',s/r}^{\alpha,s}$, $p>1, 0<r\leq s$, be  the space of all measurable functions $g$ in
	$\mathbb{R}^n$ such that  $\|g\|_{\dot{\overline{\overline{N}}}^{\alpha,s}_{p', s/r}}<+\infty$, where 
	\begin{align*}
	\|g\|_{\dot{\overline{\overline{N}}}_{p',s/r}^{\alpha,s}}=\inf_{w}\left(\int_{\mathbb{R}^n} |g|^{p'} w^{1-p'} dx   \right)^{\frac{1}{p'}},
	\end{align*}
	and the infimum is taken over all quasi-continuous weights $w\in A_{1}$ with $\|w\|_{\dot{F}^{\alpha,s}_r}\leq 1$ and $[w]_{A_1}\leq \bar{\bf c}(n,\alpha)$.    
	Then arguing as in the proof of Theorem \ref{switch} (and using Fatou's lemma to show that $[w]_{A_1}\leq \bar{\bf c}(n,\alpha)$), it follows   that $\dot{\overline{\overline{N}}}_{p',s/r}^{\alpha,s}$ is also a Banach function space.
	Moreover, as in Theorem \ref{3spa}, one has $\dot{\overline{\overline{N}}}^{\alpha,s}_{p', s/r} \approx \dot{\overline{N}}^{\alpha,s}_{p', s/r}$. 
\end{remark}

It is easy to see that 	 for any $\{g_{i}\}_{i=1}^{m}\subseteq \dot{\overline{N}}_{p',s/r}^{\alpha,s}$, $m\in\mathbb{N}$, and any weight $w$ such that $\|w\|_{\dot{F}}\leq 1$, we have 
	\begin{align*}
\left(\sum_{i=1}^{m}\|g_{i}\|_{\dot{\overline{N}}_{p',s/r}^{\alpha,s}}^{p'}\right)^{\frac{1}{p'}}\leq\left\|\left(\sum_{i=1}^{m}|g_{i}|^{p'}\right)^{\frac{1}{p'}}\right\|_{L^{p'}(\mathbb{R}^n, w^{1-p')}}.
\end{align*}
Thus, we have the following result.
\begin{proposition}\label{N-concave}
	For $s>1$, $0<\alpha<n/s$, $0<r\leq s$, and $p>1$, the space $\dot{\overline{N}}_{p',s/r}^{\alpha,s}$ is $p'$-concave with $p'$-concavity constant $1$. Thus it has an absolutely 
	continuous norm and is separable. Moreover, it holds that,  isometrically, 
	$$\left(\dot{\overline{N}}_{p',s/r}^{\alpha,s}\right)''= \dot{\overline{N}}_{p',s/r}^{\alpha,s}, \quad \left(\dot{\overline{N}}_{p',s/r}^{\alpha,s}\right)'=\left(\dot{\overline{N}}_{p',s/r}^{\alpha,s}\right)^*.$$
\end{proposition}

We have the following density result for $\dot{\overline{N}}_{p',s/r}^{\alpha,s}$.

\begin{proposition}\label{density}
	For $s>1$, $0<\alpha<n/s$, $0<r\leq s$, and $p>1$, the space $C_c^\infty(\mathbb{R}^n)$ is dense in $\dot{\overline{N}}_{p',s/r}^{\alpha,s}$.
\end{proposition}
\begin{proof}
	Note that for any $g\in \dot{\overline{N}}^{\alpha, s}_{p', s/r}$, by Lebesgue dominated convergence theorem  the functions $g_{N}=g\cdot\chi_{\{|g|\leq N\}}\cdot \chi_{\{|x|\leq N\}}, N\in \mathbb{N},$ 	converge to $g\in \dot{\overline{N}}^{\alpha, s}_{p', s/r}$, as $N\rightarrow\infty$. Thus, bounded functions with compact support $f$ are dense in $\dot{\overline{N}}^{\alpha, s}_{p', s/r}$. For such  $f$, we define $\rho_\varepsilon * f= \varepsilon^{-n}\rho(\varepsilon^{-1} \cdot)*f$, where $\varepsilon\in(0,1)$ and $\rho\in C_c^\infty(B_1(0))$. Let $B$ be a ball such that 
${\rm supp}(f)\subset B$ and ${\rm supp}(\rho_\varepsilon * f)\subset B$ for any $\varepsilon\in(0,1)$. Then take a weight $w\in \dot{F}$ such that 
$w\geq 1$ on $B$. We have
\begin{align*}
\|\rho_\varepsilon * f -f\|_{\dot{\overline{N}}^{\alpha, s}_{p', s/r}} &\leq C_w \left(\int_{B}|\rho_\varepsilon * f -f|^{p'} w^{1-p'} dx\right)^{\frac{1}{p'}}\\
&\leq C_w \|\rho_\varepsilon * f -f\|_{L^{p'}(B)}.
\end{align*}
Thus we see that  $C_c^\infty(\mathbb{R}^n)$ is dense in 	$\dot{\overline{N}}^{\alpha, s}_{p', s/r}$. 
\end{proof}

We now point out another equivalent norm for $\dot{\overline{N}}_{p',s/r}^{\alpha,s}$, but this norm will not be used  in the paper.

\begin{proposition}
	Suppose that  $s>1$, $0<\alpha<\frac{n}{s}$, $p>1$, and $0<  r \leq  s$. Then for any $g\in \dot{\overline{N}}_{p',s/r}^{\alpha,s}$, one has
	$$\|g\|_{\dot{\overline{N}}_{p',s/r}^{\alpha,s}}=\|g\|'',$$
	where 
	\begin{align*}
	\|g\|''=\inf\left\{\sum_{j}|c_{j}|:g=\sum_{j}c_{j}b_{j}~{\rm a.e.},~b_{j}\in C^\infty_{c}(\mathbb{R}^{n}),~\|b_{j}\|_{\dot{\overline{N}}_{p',s/r}^{\alpha,s}}\leq 1\right\}.
	\end{align*}
\end{proposition}	
\begin{proof} First recall from the proof of Theorem \ref{switch} that  
\begin{align*}
\|g\|_{\dot{\overline{N}}_{p',s/r}^{\alpha,s}}=\|g\|',
\end{align*}
where $\left\|\cdot\right\|'$ is given in  by
\begin{align*}
\|g\|'=\inf\left\{\sum_{j}|c_{j}|:g=\sum_{j}c_{j}b_{j}~{\rm a.e.},~\|b_{j}\|_{\dot{\overline{N}}_{p',s/r}^{\alpha,s}}\leq 1\right\}.
\end{align*}

Thus, one obtains
\begin{align}\label{NB}
\|g\|_{\dot{\overline{N}}_{p',s/r}^{\alpha,s}}\leq\|g\|''.
\end{align}

For the other direction, if $g\in C_c^\infty$, then obviously,
\begin{align}\label{BN}
\|g\|''\leq \|g\|_{\dot{\overline{N}}_{p',s/r}^{\alpha,s}}.
\end{align}
For general $g\in \dot{\overline{N}}_{p',s/r}^{\alpha,s}$, by Proposition \ref{density}, there exists a sequence $\{g_j\}\subseteq C_c^\infty$ such that $g_j\rightarrow g$ in  $\dot{\overline{N}}_{p',s/r}^{\alpha,s}$. Thus, $\{g_j\}$ is Cauchy in $(B, \left\|\cdot\right\|'')$, where 
$$B=\left\{f=\sum_j c_j b_j {\rm ~a.e.}: \{c_j\}\in \ell^1, b_{j}\in C^\infty_{c}(\mathbb{R}^{n}),~\|b_{j}\|_{\dot{\overline{N}}_{p',s/r}^{\alpha,s}}\leq 1 \right\}.$$
Since  $(B, \left\|\cdot\right\|'')$ is complete,  $g_j\rightarrow \tilde{g}$ in $(B, \left\|\cdot\right\|'')$ for some $\tilde{g}\in (B, \left\|\cdot\right\|'')$. But then by \eqref{NB} we have 
that $g_j\rightarrow \tilde{g}$ in  $\dot{\overline{N}}_{p',s/r}^{\alpha,s}$, which gives $g=\tilde{g}$. Now sending $j\rightarrow\infty$ in the inequality 
\begin{align*}
\|g_j\|''\leq \|g_j\|_{\dot{\overline{N}}_{p',s/r}^{\alpha,s}}
\end{align*}
we get \eqref{BN} as well.
\end{proof}

\begin{theorem}\label{barNMprime}
	For $s>1$, $0<\alpha<n/s$, $0<r\leq s$, and $p>1$, it holds that, isometrically, 
$$\dot{\overline{N}}_{p',s/r}^{\alpha,s}	= (\dot{M}^{\alpha,s}_{p,r})', \quad (\dot{\overline{N}}_{p',s/r}^{\alpha,s})' = \dot{M}^{\alpha,s}_{p,r}.$$
\end{theorem}
\begin{proof}
	By Proposition \ref{N-concave}, it is enough to show that $(\dot{\overline{N}}_{p',s/r}^{\alpha,s})' = \dot{M}^{\alpha,s}_{p,r}.$ On the one hand, let
	$f\in \dot{M}^{\alpha,s}_{p,r}$ and  $g\in \dot{\overline{N}}_{p',s/r}^{\alpha,s}$. Then for any $\epsilon>0$, the exists a weight $w$ such that $\|w\|_{\dot{F}}\leq 1$ and 
	$$ \left(\int_{\mathbb{R}^n} |g|^{p'} w^{1-p'}dx \right)^{\frac{1}{p'}} \leq \|g\|_{\dot{\overline{N}}_{p',s/r}^{\alpha,s}}+\epsilon.$$
	By Lemma \ref{extreme}, there exists a nonnegative $h\in L^s(\mathbb{R}^n)$ such that 
	$I_\alpha h\geq w^{\frac1r}$ q.e. and 
	$$\|h\|_{L^s(\mathbb{R}^n)}^r \leq 1.$$	
	Using H\"older's inequality, we have 
	\begin{align*}
	\int_{\mathbb{R}^{n}}|fg|dx	&\leq\left(\int_{\mathbb{R}^{n}}|f|^{p}\omega dx\right)^{\frac{1}{p}}\left(\int_{\mathbb{R}^{n}}|g|^{p'}\omega^{1-p'}dx\right)^{\frac{1}{p'}}\\
	&\leq\left(\int_{\mathbb{R}^{n}}|f|^{p}(I_{\alpha}h)^{r} dx\right)^{\frac{1}{p}}\left(\|g\|_{\dot{\overline{N}}_{p',s/r}^{\alpha,s}}+\epsilon\right)\\
	&\leq\|f\|_{\dot{M}^{\alpha,s}_{p,r}} \|h\|_{L^s(\mathbb{R}^n)}^{\frac{r}{p}}  \left(\|g\|_{\dot{\overline{N}}_{p',s/r}^{\alpha,s}}+\epsilon\right)\\
		&\leq\|f\|_{\dot{M}^{\alpha,s}_{p,r}}  \left(\|g\|_{\dot{\overline{N}}_{p',s/r}^{\alpha,s}}+\epsilon\right).
	\end{align*} 
	Letting $\epsilon\rightarrow0$, we obtain 
	$$\|f\|_{(\dot{\overline{N}}^{\alpha,s}_{p',s/r)'}} \leq\|f\|_{\dot{M}^{\alpha,s}_{p,r}}.$$

On the other hand, let
$f\in (\dot{\overline{N}}^{\alpha,s}_{p',s/r})'$ and  $h\in L^s(\mathbb{R}^n)$, $h\geq 0$. Then $(I_\alpha h)^r\in \dot{F}$ and $\| (I_\alpha h)^r\|_{\dot{F}} \leq \|h\|_{L^s(\mathbb{R}^n)}^r$. Thus, 
	\begin{align*}
&\int_{\mathbb{R}^{n}}|f|^{p} (I_\alpha h)^r dx\leq\|f\|_{(\dot{\overline{N}}_{p',s/r}^{\alpha,s})'}\cdot\left\||f|^{p-1} (I_\alpha h)^r\right\|_{\dot{\overline{N}}_{p',s/r}^{\alpha,s}}\\
&\leq\|f\|_{(\dot{\overline{N}}_{p',s/r}^{\alpha,s})'} \inf_{\|w\|_{\dot{F}}\leq 1}\left(\int_{\mathbb{R}^{n}}|f|^{(p-1)p'} (I_\alpha h)^{r p'}\cdot w^{1-p'}dx\right)^{\frac{1}{p'}}\\
&\leq\|f\|_{(\dot{\overline{N}}_{p',s/r}^{\alpha,s})'} \left(\int_{\mathbb{R}^{n}}|f|^{(p-1)p'} (I_\alpha h)^{r p'}\cdot (I_\alpha h)^{r(1-p')} \|h\|_{L^s(\mathbb{R}^n)}^{r(p'-1)} dx\right)^{\frac{1}{p'}}\\
&\leq\|f\|_{(\dot{\overline{N}}_{p',s/r}^{\alpha,s})'} \left(\int_{\mathbb{R}^{n}}|f|^{p} (I_\alpha h)^{r} dx\right)^{\frac{1}{p'}}  \|h\|_{L^s(\mathbb{R}^n)}^{\frac{r}{p}}.
\end{align*}
	Assume at the moment that $f$ is a compactly supported and bounded function. Then it follows that 
		\begin{align*}
	\int_{\mathbb{R}^{n}}|f|^{p} (I_\alpha h)^{r} dx<+\infty,
	\end{align*}
and	as a result, we have 
\begin{align*}
\left(\int_{\mathbb{R}^{n}}|f|^{p} (I_\alpha h)^r dx\right)^{\frac{1}{p}} \leq\|f\|_{(\dot{\overline{N}}_{p',s/r}^{\alpha,s})'}   \|h\|_{L^s(\mathbb{R}^n)}^{\frac{r}{p}}.
\end{align*}
This gives 
	\begin{align*}
	\|f\|_{\dot{M}^{\alpha,s }_{p,r}}\leq\|f\|_{(\dot{\overline{N}}_{p',s/r}^{\alpha,s})'}.
	\end{align*}
	For general $f\in(\dot{\overline{N}}_{p',s/r}^{\alpha,s})'$, we let $f_{N}$ be as in \eqref{FN}. Then, we have 
	\begin{align*}
	\|f\|_{\dot{M}_{p,r}^{\alpha,s}}=\sup_{N}\|f_{N}\|_{\dot{M}_{p,r}^{\alpha,s}}\leq \sup_{N}\|f_{N}\|_{(\dot{\overline{N}}_{p',s/r}^{\alpha,s})'}\leq \|f\|_{(\dot{\overline{N}}_{p',s/r}^{\alpha,s})'},
	\end{align*}
	which completes the proof.	
\end{proof}

\begin{proof}[Proof of Theorem \ref{Main3}]
	This theorem follows immediately from Theorem \ref{3spa} and Theorem \ref{barNMprime}.
\end{proof}

\begin{remark}
	Theorem \ref{3spa} and Proposition \ref{N-concave} imply that $(\dot{N}_{p',s/r}^{\alpha,s})'= (\dot{N}_{p',s/r}^{\alpha,s})^*.$ 
	 One can also use a direct argument to obtain this result as follows. Indeed, we just need to show that for any bounded linear functional $\mathcal{L}\in (\dot{N}_{p',s/r}^{\alpha,s})^*$,
one has 
$$\|f\|_{(\dot{N}_{p',s/r}^{\alpha,s})'}\leq  \|\mathcal{L}\|_{(\dot{N}_{p',s/r}^{\alpha,s})^{\ast}}$$
for a function $f\in (\dot{N}_{p',s/r}^{\alpha,s})'$ such that 
$\mathcal{L}(g)=\int_{\mathbb{R}^{n}}fgdx$ for all $g\in \dot{N}_{p',s/r}^{\alpha,s}$.

Let $g\in L^{p'}(\mathbb{R}^{n})$ be supported on a ball $B$. 
Let $h=(I_{\alpha}(\mu^{B}))^{\frac{1}{s-1}}$, where $\mu^{B}$ is the capacitary measure for $B$ (see \cite[Theorems 2.5.6]{AH}). Then 
\begin{equation*}
I_\alpha h \geq 1  \text{ q.e. on } B, \text{ and } \|h\|_{L^s}={\rm cap}_{\alpha,s}(B)^{\frac{1}{s}}. 
\end{equation*}
Note that, as   in the proof of Theorem \ref{secondtheorem}, the weight $w=I_{\alpha}(h\cdot(I_{\alpha}h)^{r-1})$ is quasi-continuous,  $[w]_{A_1}\leq \bar{\bf c}(n, \alpha)$, and  
\begin{align*}
\|w\|_{L^{s/r}({\rm cap}_{\alpha,s})}\leq c\cdot\|h\|_{L^{s}(\mathbb{R}^{n})}^{r}.
\end{align*}
Then, by Lemma \ref{IBPL},
\begin{align*}
\|g\|_{\dot{N}_{p',s/r}^{\alpha,s}}&\leq \int_{\mathbb{R}^{n}}|g|^{p'}\left(\frac{w}{c\cdot \|h\|_{L^s}^r}\right)^{1-p'} dx\\
&\leq C \cdot\|h\|_{L^{s}(\mathbb{R}^{n})}^{r(p'-1)}\int_{B}|g|^{p'}(I_{\alpha}h)^{r(1-p')}dx\\
& \leq C \cdot{\rm cap}_{\alpha,s}(B)^{\frac{r}{s(p-1)}}\|g\|_{L^{p'}(\mathbb{R}^{n})}.
\end{align*}
Thus, 
$$\left|\mathcal{L}(g) \right| \leq C \|L\|_{\dot{N}_{p',s/r}^{\alpha,s}} {\rm cap}_{\alpha,s}(B)^{\frac{r}{s(p-1)}}\|g\|_{L^{p'}(\mathbb{R}^{n})}. $$
By Riesz representation theorem,  we deduce that $\mathcal{L}$ induces a linear functional of the form
\begin{align*}
\mathcal{L}(g)=\int_{\mathbb{R}^{n}}fgdx
\end{align*}
for some $f\in L_{\rm loc}^{p}(\mathbb{R}^{n})$. Note then that 
\begin{align*}
\int_{\mathbb{R}^{n}}|f||g|dx  =\mathcal{L}(|g|{\rm sign}(f)) \leq \|\mathcal{L}\|_{(\dot{N}_{p',s/r}^{\alpha,s})^{\ast}} \|g\|_{\dot{N}_{p',s/r}^{\alpha,s}}
\end{align*}
for all $g\in L^{p'}(\mathbb{R}^n)$ with compact support. The result then follows by approximating a general function $g\in \dot{N}_{p',s/r}^{\alpha,s}$ by $g_{N}=g\cdot\chi_{\{|g|\leq N\}}\cdot \chi_{\{|x|\leq N\}}, N\in \mathbb{N}.$ 
\end{remark}

\section{Proof of Theorem \ref{Newnorm2}}

\begin{proof}[Proof of Theorem \ref{Newnorm2}]
	Let $u$ be a q.e. defined function in $\mathbb{R}^n$. Suppose that $f$ is a nonnegative measurable function such that $f\in \dot{\widetilde{O}}^{\alpha,s}_{q}$ and $I_\alpha f\geq |u|$ quasi-everywhere. Then by Theorems \ref{Main2} and \ref{Main3}, we find 
	\begin{align*}
	&\left(\int_{\mathbb{R}^n} |u|^q d {\rm cap}_{\alpha,s}\right)^{\frac{1}{q}} \leq \left(\int_{\mathbb{R}^n} (I_\alpha  f)^q d {\rm cap}_{\alpha,s}\right)^{\frac{1}{q}} \\
	& \leq A_1 \|f\|_{\dot{\widetilde{O}}^{\alpha, s}_{q}} \leq  A_2 \|f\|_{\dot{KV}_{q}} \leq A_2 \left(\int_{\mathbb{R}^n} f^s (I_\alpha f)^{q-s} dx\right)^\frac1q. 
	\end{align*}
	Now taking the infimum over such $f$ we arrive at
	$$\left(\int_{\mathbb{R}^n} |u|^q d {\rm cap}_{\alpha,s} \right)^{\frac1q} \lesssim \dot{\lambda}^{\alpha,s}_q(u) \lesssim \dot{\beta}^{\alpha,s}_q(u).$$
	
	Thus to complete the proof, we just need  to show that 
	\begin{equation}\label{nonlinearnorm}
	\dot{\beta}^{\alpha,s}_q(u) \lesssim \left(\int_{\mathbb{R}^n} |u|^q d {\rm cap}_{\alpha,s}\right)^{\frac1q}.
	\end{equation} 
	To this end, let $f\in L^s(\mathbb{R}^n), f\geq 0$, be such that $I_\alpha f\geq |u|^{\frac{q}{s}}$ q.e. Then by Lemma \ref{IBPL}, we have
	$$|u|\leq (I_\alpha f)^{\frac{s}{q}}\leq c\cdot I_\alpha [f (I_\alpha f)^{\frac{s}{q}-1}] \quad {\rm q.e.} $$
	Moreover, by \eqref{hIhO}, we have $f (I_\alpha f)^{\frac{s}{q}-1}\in \dot{\widetilde{O}}^{\alpha,s}_q$. Thus, it follows from the definition of $\dot{\beta}_{\alpha,s}(u)$  and Lemma \ref{IBPL} that 
	\begin{align*}
	\dot{\beta}^{\alpha,s}_{q}(u)&\leq c \left(\int_{\mathbb{R}^n} [f (I_\alpha f)^{\frac{s}{q}-1}]^s I_\alpha [f (I_\alpha f)^{\frac{s}{q}-1}]^{q-s} dx\right)^{\frac1q}\\
	&\leq c \left(\int_{\mathbb{R}^n}  [f (I_\alpha f)^{\frac{s}{q}-1}]^s [(I_\alpha f)^{\frac{s}{q}}]^{q-s} dx\right)^{\frac1q}\\
	& = c \left(\int_{\mathbb{R}^n} f^s dx\right)^{\frac1q}. 
	\end{align*}
	Taking the infimum over such $f$ and invoking Theorem \ref{Fnorm} we obtain \eqref{nonlinearnorm} as desired. 
\end{proof}

\section{The space $\dot{KV}_q$ and the triplet $\overline{C_{c}(\mathbb{R}^{n})}^{\dot{M}_{p,r}^{\alpha,s}}\text{--}(\dot{M}_{p,r}^{\alpha,s})'\text{--}\dot{M}_{p,r}^{\alpha,s}$}\label{triplet}

\subsection{The space $\dot{KV}_q$} We now revisit $\dot{KV}_q$ to show that it is indeed a Banach function space.
For $s>1$, $0<\alpha<\frac{n}{s}$,  and $1\leq q <  s$, we  now define 
$$\dot{\overline{O}}^{\alpha,s}_q:=\dot{\overline{N}}^{\alpha,s}_{s, s/r}, \quad r= \frac{s(s-q)}{(s-1)q}.$$
By Theorems \ref{Main3} and  \ref{3spa}, one has 
$$\dot{\overline{O}}^{\alpha,s}_q \approx \dot{\widetilde{O}}^{\alpha,s}_q \approx \dot{O}^{\alpha,s}_{q} \approx \dot{KV}_q,$$  
and so
\begin{align*}
\|f\|_{\dot{KV}_q}=\inf\left\{\left(\int_{\mathbb{R}^{n}}h^{s}(I_{\alpha}h)^{q-s}dx\right)^{\frac{1}{q}}:h\in \dot{\overline{O}}_{q}^{\alpha,s},~h\geq|f|~{\rm a.e.}\right\}.
\end{align*}	

\begin{lemma}\label{KV-semi} For $s>1$, $0<\alpha<\frac{n}{s}$,  and $1\leq q <  s$, it holds that, for any $g\in \dot{KV}_q$,  
	\begin{align*}
	\|g\|_{\dot{KV}_q}=\inf\left\{\sum_{j}|c_{j}|:g=\sum_{j}c_{j}b_{j}~{\rm a.e.},~\|b_{j}\|_{\dot{KV}_q}\leq 1\right\}.
	\end{align*}
\end{lemma}
\begin{proof} Let $g\in \dot{KV}_q$. Then obviously,
	$$\inf\left\{\sum_{j}|c_{j}|:g=\sum_{j}c_{j}b_{j}~{\rm a.e.},~\|b_{j}\|_{\dot{KV}_q}\leq 1\right\} \leq \|g\|_{\dot{KV}_q}.$$ 
On the other hand, suppose that $g$ admits a decomposition of the form
$$g=\sum_{j}c_{j}b_{j}\quad {\rm a.e.},$$
 where $\|b_{j}\|_{\dot{KV}_q}\leq 1$ and $\{c_j\}\in \ell^1$. Then for any $\varepsilon>0$ and $j\geq 1$, there exists  $h_j \in \dot{\overline{O}}_{q}^{\alpha,s}$ such that  $h_j\geq  |b_j|$ a.e. and 
\begin{equation}\label{hjIhj}
\left(\int_{\mathbb{R}^{n}} h_j^{s}(I_{\alpha}h_j)^{q-s}dx\right)^{\frac{1}{q}} \leq \|b_{j}\|_{\dot{KV}_q}+\varepsilon \leq 1+\varepsilon.
\end{equation}

This gives
 $$\|h_{j}\|_{\dot{KV}_q} \leq 1+\varepsilon.$$
Let $h= \sum_{j}|c_{j}| h_{j}$. Then $h\geq |g|$ a.e. and moreover,
\begin{align*}
\|h\|_{\dot{\overline{O}}_{q}^{\alpha,s}} &\leq \sum_{j}|c_{j}| \|h_{j}\|_{\dot{\overline{O}}_{q}^{\alpha,s}}\\
&\leq C \sum_{j}|c_{j}| \|h_{j}\|_{\dot{KV}_q}\\
& \leq C (1+\varepsilon) \sum_{j}|c_{j}| <+\infty.
\end{align*}
Thus, 
\begin{equation}\label{KVhIh}
\|g\|_{\dot{KV}_q}\leq \left(\int_{\mathbb{R}^{n}} h^{s}(I_{\alpha}h)^{q-s}dx\right)^{\frac{1}{q}}.
\end{equation}

Let $w_j= (I_\alpha h_j)^{\frac{s-q}{s-1}}$  and 
	\begin{align*}
w=\|\{c_{j}\}\|_{\ell^{1}}^{-1}\sum_{j}|c_{j}|\cdot w_{j}.
\end{align*}
 By H\"older's inequality,
 $$w\leq \left(\|\{c_{j}\}\|_{\ell^{1}}^{-1} \sum_{j}|c_{j}|\cdot w_{j}^{\frac{s-1}{s-q}}\right)^{\frac{s-q}{s-1}}=\|\{c_{j}\}\|_{\ell^{1}}^{\frac{q-s}{s-1}}\left(\sum_{j}|c_{j}|\cdot w_{j}^{\frac{s-1}{s-q}}\right)^{\frac{s-q}{s-1}}.$$
 Thus,
\begin{align*} 
h^{s} (I_\alpha h)^{q-s}&= h^s \left(\sum_{j} |c_j| I_{\alpha} h_j \right)^{q-s} = h^s \left(\sum_{j} |c_j| w_j^{\frac{s-1}{s-q}} \right)^{q-s}\\
&\leq  \|\{c_j\}\|_{\ell^1}^{q-s} h^s w^{1-s} \\
& \leq \|\{c_{j}\}\|_{\ell^{1}}^{q-s} \|\{c_{j}\}\|_{\ell^{1}}^{s-1}\left(\sum_{j}|c_{j}|\cdot  h_{j}^{s} w_{j}^{1-s} \right)\\
&= \|\{c_{j}\}\|_{\ell^{1}}^{q-1}  \sum_{j}|c_{j}|\cdot  h_{j}^{s} (I_\alpha h_j)^{q-s}, 
\end{align*}
where the last inequality follows as in \eqref{gpw1p}. From this and \eqref{hjIhj}, we get
\begin{align*} 
\int_{\mathbb{R}^n}h^{s} (I_\alpha h)^{q-s}&\leq\|\{c_{j}\}\|_{\ell^{1}}^{q-1}\left(\sum_{j}|c_{j}| \int_{\mathbb{R}^n}  h_{j}^{s} (I_\alpha h_j)^{q-s} dx \right)\\
&\leq \|\{c_{j}\}\|_{\ell^{1}}^{q} (1+\varepsilon)^q. 
\end{align*}
Then letting $\varepsilon\rightarrow 0$ and using \eqref{KVhIh}, this yields that 
$$\|g\|_{\dot{KV}_q}\leq \|\{c_{j}\}\|_{\ell^{1}}.$$
Finally, taking the infimum over such decompositions of $g$ we arrive at the bound
$$\|g\|_{\dot{KV}_q}\leq  \inf\left\{\sum_{j}|c_{j}|:g=\sum_{j}c_{j}b_{j}~{\rm a.e.},~\|b_{j}\|_{\dot{KV}_q}\leq 1\right\},$$
which completes the proof.
\end{proof}

\begin{theorem} For $s>1$, $0<\alpha<\frac{n}{s}$,  and $1\leq q <  s$, the space $\dot{KV}_q$ is a Banach function space.
\end{theorem}
\begin{proof} By Lemma \ref{KV-semi}, $\left\|\cdot\right\|_{\dot{KV}_q}$ is a semi-norm and thus $\dot{KV}_q$ is a Banach space in view of Theorem \ref{Main3}. 
	
	Properties (P3) and (P4)  in the definition of Banach function space also follow from Theorem \ref{Main3}.
	Now to show  the Fatou's property (P2) for $\dot{KV}_q$, let $0\leq g_{j}\uparrow g$ and $g_{j}\in \dot{KV}_q$ with $\sup_{j}\|g_{j}\|_{\dot{KV}_q}<+\infty$. Since $g_{j}\in  \left(\dot{M}^{\alpha, s}_{s', r}\right)'$, $r=\frac{s(s-q)}{(s-1)q}$, and $\left(\dot{M}^{\alpha, s}_{s', r}\right)'$ has the Fatou's property, it follows that $g\in  \left(\dot{M}^{\alpha, s}_{s',r}\right)'$ which in turn implies that $g\in \dot{KV}_q$. In particular, $g(x)<+\infty$ a.e. and hence $|g_j-g| \downarrow 0$ a.e. Since  $\left(\dot{M}^{\alpha, s}_{s',r}\right)'$ has an absolute continuous norm (by Proposition \ref{M'*}), the same is also true for $\dot{KV}_q$. Then, it holds  that $\|g_{j}-g\|_{\dot{KV}_q}\rightarrow 0$ and hence by triangle inequality,
	\begin{align*}
	\|g\|_{\dot{KV}_q}\leq \lim_{j\rightarrow\infty}\|g_{j}\|_{\dot{KV}_q}\leq\sup_{j}\|g_{j}\|_{\dot{KV}_q},
	\end{align*}
which yields the property (P2).
\end{proof}

\subsection{The triplet $\overline{C_{c}(\mathbb{R}^{n})}^{\dot{M}_{p,r}^{\alpha,s}}\text{--}(\dot{M}_{p,r}^{\alpha,s})'\text{--}\dot{M}_{p,r}^{\alpha,s}$}
Suppose that  $s>1$, $0<\alpha<\frac{n}{s}$, $p>1$, and $0<  r \leq  s$.
Let $E$ be a measurable set in $\mathbb{R}^n$ with finite measure. By Theorem \ref{upper-tri} and the property $|K|^{1-\frac{\alpha s}{n}}\leq c\cdot {\rm cap}_{\alpha,s}(K)$ for any measurable set $K\subseteq\mathbb{R}^n$,  it is easy to deduce that 
\begin{align*}
\|f\cdot\chi_{E}\|_{\dot{M}_{p,r}^{\alpha,s}}\leq c\cdot|E|^{(1-\frac{r}{s}+\frac{\alpha r}{n})\frac{1}{p}}\cdot\|f\|_{L^{\infty}(E)}.
\end{align*}
In particular, 
\begin{align*} 
\|\chi_{E}\|_{\dot{M}_{p,r}^{\alpha,s}}\leq c\cdot|E|^{(1-\frac{r}{s}+\frac{\alpha r}{n})\frac{1}{p}},
\end{align*}
and for measurable sets $B\subseteq A$,  
\begin{align*} 
\|\chi_{A} -\chi_{B}\|_{\dot{M}_{p,r}^{\alpha,s}} = \|\chi_{A \setminus B} \|_{\dot{M}_{p,r}^{\alpha,s}} \leq c\cdot|A\setminus B|^{(1-\frac{r}{s}+\frac{\alpha r}{n})\frac{1}{p}}.
\end{align*}

Let $\dot{\mathcal{F}}$ be the closure in $\dot{M}^{\alpha,s}_{p,r}$ of the set of all finite linear combinations of  characteristic functions of sets of finite Lebesgue measure.
Then one can adapt the argument in the proof of \cite[Theorem 4.3]{ST} to obtain
\begin{align*}
(\dot{\mathcal{F}})^{\ast} = (\dot{M}_{p,r}^{\alpha,s})'
\end{align*}
in the sense that each bounded linear functional $L$ on $\dot{\mathcal{F}}$ corresponds to a unique $g\in (\dot{M}_{p,r}^{\alpha,s})'$ such that 
$$L(f) = \int_{\mathbb{R}^n} g f dx, \qquad \forall f\in \dot{\mathcal{F}},$$
and $\|L\|_{(\dot{\mathcal{F}})^*}=\|g\|_{(\dot{M}_{p,r}^{\alpha,s})'}$. With this, we can argue as in the proof of \cite[Theorem 1.9]{OP1}, which relies on  Hahn-Banach theorem, to obtain
\begin{align*}
\left( \overline{C_{c}(\mathbb{R}^{n})}^{\dot{M}_{p,r}^{\alpha,s}} \right)^{\ast} = (\dot{M}_{p,r}^{\alpha,s})'.
\end{align*}

 Thus, by Proposition \ref{M'*}, we have the triplet  
\begin{align*}
\overline{C_{c}(\mathbb{R}^{n})}^{\dot{M}_{p,r}^{\alpha,s}}\text{--}(\dot{M}_{p,r}^{\alpha,s})'\text{--}\dot{M}_{p,r}^{\alpha,s},
\end{align*}
in the sense of \eqref{2dual},  which reminisces the classical triplet 
\begin{align*}
VMO=\overline{C_{c}(\mathbb{R}^{n})}^{BMO}\text{--}H^{1}(\mathbb{R}^{n})\text{--}BMO.
\end{align*}

\section{The inhomogeneous case} \label{inhomo}
Our approach here also works well for the inhomogeneous case,  
where the  space of Riesz potentials   
$\dot{H}^{\alpha, s}(\mathbb{R}^n)$ and Riesz capacity ${\rm cap}_{\alpha,s}$ are replaced, respectively, by the space of Bessel potentials $H^{\alpha, s}(\mathbb{R}^n)$ and Bessel capacity ${\rm Cap}_{\alpha,s}$, $s>1, 0<\alpha\leq \frac{n}{s}$. For $s>1$ and $\alpha>0$, recall that  the space of Bessel potentials $H^{\alpha, s}=H^{\alpha, s}(\mathbb{R}^n)$, is defined as  
 the completion of $C_c^\infty(\mathbb{R}^n)$ with respect to the norm 
$$\|u\|_{H^{\alpha,s}}=\|(1-\Delta)^{\frac{\alpha}{2}}u \|_{L^s(\mathbb{R}^n)}.$$ 
As $\alpha>0$, it follows  that (see, e.g., \cite{MH}) a function $u$ belongs to $H^{\alpha,s}$ if and only if
$$u= G_{\alpha}f:=\int_{\mathbb{R}^n} G_\alpha(\cdot-y) f(y) dy$$ 
for some $f\in L^s$, and moreover 
$\|u\|_{H^{\alpha, s}}=\|f\|_{L^s(\mathbb{R}^n)}.$ 
Here $G_\alpha$ is  the Bessel kernel of order $\alpha$ defined by $G_{\alpha}(x):= \mathcal{F}^{-1}[(1+|\xi|^2)^{\frac{-\alpha}{2}}](x)$. 
The Bessel potential  space $H^{\alpha,s}$, $\alpha>0, s>1$, can be viewed as a fractional generalization  of the standard Sobolev space $W^{k,s}=W^{k,s}(\mathbb{R}^n)$, $k\in \mathbb{N}, s>1$.  

The capacity associated to $H^{\alpha,s}$ is the Bessel capacity defined for any set $E\subset\mathbb{R}^n$ by 
\begin{equation*}
{\rm Cap}_{\alpha, \,s}(E):=\inf\Big\{\|f\|_{L^{s}(\mathbb{R}^n)}^{s}: f\geq0, G_{\alpha}f\geq 1 {\rm ~on~} E \Big\}.
\end{equation*}
As for ${\rm cap}_{\alpha,s}$, the Choquet integral associated to ${\rm Cap}_{\alpha,s}$ of  a function $g:\mathbb{R}^n \rightarrow [0,\infty]$  is defined   by
\begin{equation*}
\int_{\mathbb{R}^n} g d {\rm Cap}_{\alpha, s} :=\int_{0}^{\infty}{\rm Cap}_{\alpha,s}(\{x\in\mathbb{R}^n: g(x)>t\})dt.
\end{equation*}
Recall that the capacitary strong type inequality for Bessel potentials is given by 
\begin{equation}\label{CSI-Bes}
\int_{\mathbb{R}^n}  (G_\alpha f)^{s} d{\rm Cap}_{\alpha,s} \leq C \int_{\mathbb{R}^n} f^{s} dx  
\end{equation}
where $s>1$ and  $\alpha>0$ (see \cite{MSh, AH}).  
Thus,  the trace inequality 
\begin{equation*}
\int_{\mathbb{R}^n}  (G_\alpha f)^{s} d\mu \leq C \int_{\mathbb{R}^n} f^{s} dx  
\end{equation*}
holds for all nonnegative  $f\in L^s(\mathbb{R}^n)$ if and only if the measure $\mu$ satisfies the capacitary condition 
\begin{equation*} 
\mu(K)\leq C {\rm Cap}_{\alpha,s}(K)
\end{equation*}
for all compact sets $K\subset\mathbb{R}^n$ (actually, we can also assume that ${\rm diam}(K)\leq 1$; see \cite[Remark 3.1.1]{MSh}). 

It is important to know that the ``integrating by parts" lemma also holds for Bessel potentials, which was established recently in \cite[Remark 2.6]{GV}.
 \begin{lemma}  Let $t\geq 1$ and suppose that $f$ is a nonnegative measurable function in $\mathbb{R}^n$. Then we  have
		$$(G_\alpha f)^t \leq A G_\alpha[ f(G_\alpha f)^{t-1}]$$
		everywhere in $\mathbb{R}^n$.
	\end{lemma}

The analogue of D. R. Adams  conjecture in this inhomogeneous setting is the validity of the bound 
 \begin{equation}\label{capstrong2-Bes} 
 \int_{\mathbb{R}^n} (G_\alpha  f)^q d{\rm cap}_{\alpha, s} \leq  A \int_{\mathbb{R}^n} f^s (G_\alpha  f)^{q-s}  dx 
\end{equation}
at least for nonnegative functions $f\in L^\infty(\mathbb{R}^n)$ with compact support and $q\geq 1$. The inhomogeneous versions of 
the spaces $\dot{\widetilde{O}}^{\alpha,s}_q$, $\dot{K}_q$,  $\dot{M}^{\alpha,s}_{p,r}$, and $\dot{\widetilde{N}}^{\alpha,s}_{p',s/r}$ are denoted by ${\widetilde{O}}^{\alpha,s}_q$, $K_q$, and $M^{\alpha,s}_{p,r}$, and $\widetilde{N}^{\alpha,s}_{p',s/r}$, respectively, i.e., without the dot on top. They are defined similarly but now $I_\alpha$ is replaced by $G_\alpha$ and ${\rm cap}_{\alpha,s}$ is replaced by ${\rm Cap}_{\alpha,s}$. With this, one finds that the analogues of Theorems \ref{Main1}--\ref{KV-q} with $s>1, 0<\alpha\leq \frac{n}{s}$, are also available in this inhomogeneous setting.

On the other hand, one should be careful about the definition of $N^{\alpha,s}_{p',s/r}$, $s>1, 0<\alpha\leq \frac{n}{s}, p>1, 0<r\leq s$. It is defined similarly to that of $\dot{N}^{\alpha,s}_{p',s/r}$, but this time the $A_1$ condition $[w]_{A_1}\leq \bar{\bf c}(n,\alpha)$ on the weight $w$ is replaced by the $A_1^{\rm loc}$ condition
$$[w]_{A_1^{\rm loc}}\leq \bar{\bf c}(n,\alpha),$$
and moreover, here one can take  $\bar{\bf c}(n,\alpha)=[G_\alpha]_{A_1^{\rm loc}}$. We recall that  $A_1^{\rm loc}$ is the class of weight functions $w$ in $\mathbb{R}^n$ such that
\begin{equation*}
{\bf M}^{\rm loc} w(x)\leq C w(x) \qquad {\rm a.e.}
\end{equation*}
 The $A_1^{\rm loc}$ characteristic constant of $w$, $[w]_{A_1^{\rm loc}}$, is defined as the least possible constant $C$ in the above inequality.
The operator ${\bf M}^{\rm loc}$ is the truncated (center)  Hardy-Littlewood maximal function defined for each $f\in L^1_{\rm loc}(\mathbb{R}^n)$ by 
\begin{equation*}
{\bf M}^{\rm loc} f (x)= \sup_{0<r\leq 1}   \frac{1}{|B_r(x)|} \int_{B_{r}(x)} |f(y)|dy.
\end{equation*}

Note that by \eqref{CSI-Bes}, Theorem \ref{Fnorm} also holds for $L^{s/r}({\rm Cap}_{\alpha,s})$ provided we replace $\dot{F}$ with $F=F^{\alpha,s}_r$, $s>1, \alpha>0, 0<r\leq s$, where 
$$\|u\|_{F}:=\inf\{ \|f\|_{L^s(\mathbb{R}^{n})}^r: f\geq 0, \, G_{\alpha} f\geq |u|^{\frac1r} {\rm ~q.e.~with~respect~to~} {\rm Cap}_{\alpha,s}\}.$$		
  Using the space $F$ above we can define $\overline{N}^{\alpha,s}_{p', s/r}$ as we did for $\dot{\overline{N}}^{\alpha,s}_{p', s/r}$, except now that  $\dot{F}$ is replaced by $F$, of course. Then, as in Proposition \ref{N-concave},  $\overline{N}^{\alpha,s}_{p', s/r}$, $s>1, 0<\alpha \leq \frac{n}{s}, p>1, 0<r\leq s$, is a Banach functions space, and 
$$\left(\overline{N}_{p',s/r}^{\alpha,s}\right)''= \overline{N}_{p',s/r}^{\alpha,s}, \quad \left(\overline{N}_{p',s/r}^{\alpha,s}\right)'=\left(\overline{N}_{p',s/r}^{\alpha,s}\right)^*.$$

To obtain the analogue of Theorem \ref{upper-tri}, we need to use the truncated Wolff potential  $W_{\alpha,s}^{R} \mu$ of a nonnegative measure $\mu$ defined,  for $R>0, s>1$, and $\alpha>0$, by
$$W_{\alpha, s}^R\mu(x):= \int_0^R \left[\frac{\mu(B_t(x))}{t^{n-\alpha s}}\right]^{\frac{1}{s-1}} \frac{dt}{t}, \qquad x\in\mathbb{R}^n.$$
Note that for truncated Wolff  potentials, we have the following variant of boundedness principle:
\begin{align*}
W_{\alpha,s}^{R}\mu(x)\leq 2^{\frac{n-\alpha s}{s-1}}\cdot\sup\left\{W_{\alpha,s}^{2R}\mu(y):y\in{\rm supp}(\mu)\right\},\quad x\in\mathbb{R}^{n}.
\end{align*}
Indeed, let $x\notin{\rm supp}(\mu)$ and $x_{0}\in {\rm supp}(\mu)$ be the point that minimizes the distance from $x$ to ${\rm supp}(\mu)$. If $B_{t}(x)\cap {\rm supp}(\mu)\ne\emptyset$, then $t>|x-x_{0}|$, which in turn implies that $B_{t}(x)\subseteq B_{2t}(x_{0})$. Consequently,
\begin{align*}
W_{\alpha,s}^{R}\mu(x)	\leq\int_{0}^{R}\left(\frac{\mu(B_{2t}(x_{0}))}{t^{n-\alpha s}}\right)^{\frac{1}{s-1}}\frac{dt}{t}=2^{\frac{n-\alpha s}{s-1}}\cdot W_{\alpha,s}^{2R}\mu(x_{0}),
\end{align*}
which proves the claim.

The following  result is  analogous to Theorem \ref{upper-tri}.

\begin{theorem}\label{upper-tri-Bes} Let $s>1$, $0<\alpha \leq \frac{n}{s}$, and $0<r<s$. Then the following statements are equivalent for a nonnegative locally finite (Borel) measure $\mu$ in $\mathbb{R}^n$.

	$\rm{(i)}$ There exists a constant $A_1>0$ such that the inequality  
	\begin{equation*}
	\int_{\mathbb{R}^n} ({G}_\alpha  h) ^r  d\mu \leq A_1 \|h\|_{L^s(\mathbb{R}^n)}^{r}
	\end{equation*} 
	holds for all nonnegative $h\in L^s(\mathbb{R}^n)$.

	$\rm{(ii)}$ $\mu$ is continuous w.r.t ${\rm Cap}_{\alpha,s}$ and for any quasi-continuous or $\mu$-measurable function $u\in L^{\frac{s}{r}}({\rm Cap}_{\alpha,s})$, we have 
	\begin{equation*} 
	\left|\int_{\mathbb{R}^n} u d\mu\right| \leq A_2 \|u\|_{L^{\frac{s}{r}}({\rm Cap}_{\alpha,s})} 
	\end{equation*}
	for a constant $A_2>0$.
	
	$\rm{(iii)}$ $$[W_{\alpha,s}^{1} \mu]^{s-1}\in L^\frac{r}{s-r}(d\mu).$$
	
	$\rm{(iv)}$  $$ [W_{\alpha ,s}^{1}\mu]^{s-1} \in L^{\frac{s}{s-r}}({\rm Cap}_{\alpha,s}).$$
	Moreover, we have 
	$$A_1\simeq A_2 \simeq \left \|[W_{\alpha,s}^{1} \mu]^{s-1}\right \|_{L^\frac{r}{s-r}(d\mu)}^{\frac{r}{s}} \simeq \left \|[W_{\alpha ,s}^1\mu]^{s-1} \right\|_{L^{\frac{s}{s-r}}({\rm Cap}_{\alpha,s})}.$$	
\end{theorem}
The proof of Theorem \ref{upper-tri-Bes} is analogous to that of Theorem \ref{upper-tri}. We mention here that for that one  needs the inhomogeneous variants of 
Theorems 1.11 and 1.13 in \cite{Ver1},  which in turn rely on a theorem of Kerman and Sawyer \cite{KS}. See also \cite{HW} and a remark on page 27 of \cite{MSh}  regarding  the result of \cite{KS}. 

Now using Theorem \ref{upper-tri-Bes} and the discussion above, one can obtain the homogeneous versions of Theorems \ref{Main3},  \ref{Newnorm2}, and the duality triplet
\eqref{2dual}. Moreover, needless to say, other related results such as Theorems 
\ref{secondtheorem} and \ref{3spa}, etc. are also available, with obvious modifications, in this homogeneous case.

Finally, we mention that inequality \eqref{capstrong2-Bes} for the case $q=1$ was obtained in \cite[Theorem 1.1]{OP2},  and the equivalences 
\begin{equation*} 
\|u\|_{L^1({\rm Cap}_{\alpha,s})}\simeq \lambda^{\alpha,s}_1(u) \simeq \beta^{\alpha,s}_1(u)
\end{equation*}  
were proved in \cite[Theorem 1.2]{OP2}, with  different proofs. However, we would like to take this opportunity make a correction that in  the statement of  \cite[Theorem 1.1]{OP2} one needs to require that $f\in (M^{\alpha,s}_{s'})'$, where $(M^{\alpha,s}_{s'})'=(M^{\alpha,s}_{s',s})'$. Likewise, the conditions  $f\in (M^{\alpha,s}_{s'})'$ and $h\in (M^{\alpha,s}_{s'})'$ must also be added to the definitions of    $\beta_{\alpha,s}(u)$ (p. 591) and of  $\|f\|_{KV}$ (p. 594) of \cite{OP2}, respectively. 

\bigskip
\noindent {\bf Acknowledgements:} N. C. Phuc is supported in part by Simons Foundation [Record ID: MPS-TSM-00002686].

\bigskip 

\noindent {\bf Declaration of interests:} We have nothing to declare.

\end{document}